\documentclass{amsart}

\usepackage{hyperref}

\newcommand{\norm}[2][]{\|{#2}\|_{{#1}}}

\newcommand{\EE}{{\mathbf E}}
\newcommand{\dd}{ { \mathrm{d}} }
\newcommand{\ee}{ { \mathrm{e}} }
\DeclareMathOperator*{\Tr}{\mathrm{Tr}}

\newcommand{\cF}{{ \mathcal F}}
\newcommand{\cD}{{ \mathcal D}}

\newcommand{\IR}{{ \mathbf R}}

\newcommand{\HS}{\mathrm{HS}}

\newcommand{\ts}{k}
\newcommand{\tsa}{\delta}
\newcommand{\tsb}{\gamma}
\newcommand{\hts}{\tfrac{\ts}{2}}
\newcommand{\Rk}{R_{\frac{\ts}{2}}}
\newcommand{\Ra}{R_{\frac{\tsa}2}}
\newcommand{\Rb}{R_{\frac{\tsb}2}}
\newcommand{\Jk}{J_{\ts}}
\newcommand{\Ja}{J_{\tsa}}
\newcommand{\Jb}{J_{\tsb}}
\newcommand{\Ak}{\mathcal{A}_{\ts}}
\newcommand{\Aa}{\mathcal{A}_{\tsa}}
\newcommand{\Ab}{\mathcal{A}_{\tsb}}
\newcommand{\mm}{M}
\newcommand{\Mk}{\mm_{\ts}}
\newcommand{\tMk}{\tilde{\mm}_{\ts}}
\newcommand{\tMa}{\tilde{\mm}_{\tsa}}
\newcommand{\tMb}{\tilde{\mm}_{\tsb}}
\newcommand{\Ma}{\mm_{\tsa}}
\newcommand{\Mb}{\mm_{\tsb}}
\newcommand{\fk}{f_{\ts}}
\newcommand{\fa}{f_{\tsa}}
\newcommand{\fb}{f_{\tsb}}
\newcommand{\Fk}{F_{\ts}}
\newcommand{\Fa}{F_{\tsa}}
\newcommand{\Fb}{F_{\tsb}}
\newcommand{\Xt}{\tilde{X}}
\newcommand{\Xtk}{\Xt_{\ts}}
\newcommand{\Xta}{\Xt_{\tsa}}
\newcommand{\Xtb}{\Xt_{\tsb}}

\newcommand{\XX}{X}
\newcommand{\Xj}{\XX^j}
\newcommand{\Xjm}{\XX^{j-1}}
\newcommand{\Yj}{Y^j}
\newcommand{\Zj}{Z^j}
\newcommand{\Wj}{\Delta W^j}
\newcommand{\Xb}{\bar{X}}
\newcommand{\Xh}{\hat{X}}
\newcommand{\Xbe}{X_{\mathrm{be}}}
\newcommand{\Xbej}{\Xbe^j}

\newcommand{\dH}{\dot{H}}

\newcommand{\Ek}{E_k}

\newcommand{\Wk}{W_{\Ak}}
\newcommand{\WA}{W_A}
\newcommand{\HSAQ}{\|A^{1/2}Q^{1/2}\|_{\HS}}

\newcommand{\tY}{\tilde{Y}}

\newcommand{\Ordo}{\mathcal{O}}

\newcommand{\R}{\mathbb{R}}

\newcommand{\Dom}{\cD}

\newtheorem{thm}{Theorem} [section]
\newtheorem{corollary} [thm] {Corollary}
\newtheorem{lem} [thm] {Lemma}
\newtheorem{prop} [thm] {Proposition}
\newtheorem{rem} [thm] {Remark}
\newtheorem{assumption}[thm]{Assumption}
\theoremstyle {definition}

\numberwithin{equation}{section}

\begin{document}


\title[]{On the discretisation in time of the stochastic Allen--Cahn
  equation}

\author[M.~Kov\'acs]{Mih\'aly Kov\'acs}
						
\address{Department of Mathematics and Statistics,
  University of Otago, P.O.~Box 56, Dunedin 9054, New Zealand} 

\author[S.~Larsson]{Stig Larsson}
\address{Department of Mathematical Sciences,
  Chalmers University of Technology and University of Gothenburg,
  SE--412 96 Gothenburg,
  Sweden}

\author[F.~Lindgren]{Fredrik Lindgren}
\address{Cybermedia Center, Osaka University, 1--32 Machikaneyama,
  Toyonaka, Osaka 560--0043, Japan}

\begin{abstract}
  We consider the stochastic Allen--Cahn equation perturbed by smooth
  additive Gaussian noise in a bounded spatial domain with smooth boundary in
  dimension $d\le 3$, and study the semidiscretisation in time of the
  equation by an Euler type split-step method with step size $k>0$. We show that the method
  converges strongly with a rate $O(k^{\frac12})$. By means of
  a perturbation argument, we also establish the strong convergence of
  the standard backward Euler scheme with the same rate.
\end{abstract}

\maketitle                  

\section{Introduction}

Let $\mathcal{D}\subset \mathbb{R}^d$, $d\le 3$, be a bounded spatial domain
with smooth boundary $\partial\cD$ and consider the stochastic partial
differential equation written in the abstract It\^o form
\begin{equation}\label{eq:sac}
\dd X(t)+AX(t)\,\dd t+f(X(t))\,\dd t=\dd W(t),~t\in(0,T];\quad X(0)=X_0,
\end{equation}
where $\{W(t)\}_{t\ge 0}$ is an $L_2(\cD)$-valued $Q$-Wiener process
on a filtered probability space
$(\Omega, \cF, \mathbb{P}, \{\cF_t\}_{t\ge 0})$ with respect to the
normal filtration $ \{\cF_t\}_{t\ge 0}$.  We use the notation
$H=L_2(\cD)$ with inner product $\langle\cdot\,,\cdot\rangle$ and
induced norm $\norm{\cdot}$ and $V=H^1_0(\cD)$ with inner product
$\langle\cdot\,,\cdot\rangle_1:=\langle \nabla \cdot\,,\nabla
\cdot\rangle$
and induced norm $\norm[1]{\cdot}$. Moreover, let
$Au=- \Delta u$ with $D(A)=H^2(\mathcal D)\cap H^1_0(\mathcal{D})$,
that is, the Dirichlet Laplacian on $H$.
We denote by $\{E(t)\}_{t\ge 0}$ the analytic semigroup in $H$
generated by $-A$.  Finally, $f\colon D_f\subset H\to H$ is a
Nemytskij operator $f(u)(s)=P(u(s))$, $s\in \mathcal{D}$, with
$P(r)=r^3-\beta^2r$; that is, $P=F'$ where
$F(r)=\frac14(r^2-\beta^2)^2$ is a double well potential. Note that
$f$ is only locally Lipschitz and does not satisfy a linear growth
condition. It does, however, satisfy a global one-sided Lipschitz
condition \eqref{eq:onesided1} and a local Lipschitz condition
\eqref{eq:local}, see below. Since $P$ is a third degree polynomial
and the spatial dimension $d\leq 3$, we have, by Sobolev's inequality,
that $H^1_0(\cD) \subset L_6(\cD)\subset D_f$.

We consider a fully implicit split-step scheme for the temporal
discretisation of \eqref{eq:sac} via the iteration, for $j=0,1,\dots, N-1$,
\begin{align}
X^0&=X_0,\label{eq:ss10}\\
Y^j+\hts A Y^j&= X^j,\label{eq:ss1a}\\
Z^j+\ts f(Z^j)&= Y^j,\label{eq:ss1b}\\
X^{j+1}+\hts A X^{j+1}&=Z^j+\Delta W^j,
\label{eq:ss1c}
\end{align}
where $k=T/N$ is a step size with $2\beta^2\ts<1$, $t_j=jk$, and
$\Delta W^j=W(t_{j+1})-W(t_j)$. This scheme is implicit also in the
drift term $f$, but, to compute $X^{j+1}$ one only needs to evaluate
$f$ at the $\cF_{t_j}$-measurable function $Z^j$. This is a key point,
since it allows the construction of an It\^o-process and the use of
It\^o's formula in the analysis.


The scheme in \eqref{eq:ss10}--\eqref{eq:ss1c} is an attempt to
generalise the method of Higham, Mao, and Stuart \cite{HMS} for SDEs
to the infinite-dimensional case. The specific form of the
non-linearity, $f(u)(s)=u(s)^3-\beta^2u(s)$, $s\in \mathcal{D}$,
considered here, is only used when proving the existence of solutions
of \eqref{eq:ss1b} in Lemma~\ref{lem:fprop}. Apart from this, the
properties of $f$ that will be utilised are the following one-sided
Lipschitz condition, local Lipschitz condition, and polynomial growth
condition:
\begin{align}
\langle f(x)-f(y),x-y \rangle&\geq -\beta^2\|x-y\|^2,\quad x\in H,\label{eq:onesided1}\\
\|A^{-1/2}(f(x)-f(y))\|&\le P^1(\|x\|_1,\|y\|_1)\|x-y\|,\quad x\in V,\label{eq:local}\\
\|f(x)\|&\leq P^2(\|x\|_1), \quad x\in V, \label{eq:fgrowth}
\end{align}
where $P^1$ and $P^2$ are  polynomials. In our specific case,
\begin{equation*}
P^1(s,t)=C_1(t^2+s^2+1), \quad
P^2(s)=C_2(s^3+1),
\end{equation*}
for some sufficiently large numbers $C_1$ and $C_2$.  We also use the
fact that $f(0)=0$ and the one-sided growth properties
\begin{align}\label{eq:diss1}
\langle f(x),x\rangle&\geq -\beta^2\|x\|^2, \quad x\in H,
\\
\label{eq:diss2}
\langle f(x),x\rangle_1&\geq -\beta^2\|x\|_1^2,\quad x\in V.
\end{align}



The motivation for the three-step scheme is that it has a symmetry
property with respect to $f$ that allows us to view it as the
Euler--Maruyama scheme for a perturbed equation, see \eqref{eq:pspde},
where $A$ is replaced by a bounded, self-adjoint, positive definite
operator $\Ak$, and $f$ is replaced by an operator $\Fk$ that is
globally Lipschitz continuous and retains the properties
\eqref{eq:onesided1} and \eqref{eq:diss1}--\eqref{eq:diss2}. The
smoothing of the operators allows us to use It\^o's formula, paving
the way for a situation where the one-sided Lipschitz condition helps
us to prove moment bounds and establish the correct convergence rate.
In contrast to, for example, \cite{KLLAC}, we thus avoid the use of
the mild form of \eqref{eq:sac}, which only allows for a proof of the
mere fact of strong convergence without rate. How the symmetry comes
into play will be clear from the proof of Item~\eqref{it:oneside} of
Lemma~\ref{lem:fprop}.

Given that the solution of \eqref{eq:ss10}--\eqref{eq:ss1c} converges
to the solution of \eqref{eq:sac} it is not difficult to prove that
the solution of the standard fully implicit backward Euler (BE) scheme
for \eqref{eq:sac},
\begin{equation}\label{eq:be}
\Xbej-\Xbe^{j-1}+kA\Xbej+kf(\Xbej)=\Delta W^{j-1},
\  j=1,2,\ldots,N;\quad \Xbe^0=X_0.
\end{equation}
also converges, and with the same rate. The main findings of this
paper are summarised in the following theorem.  We use the notation
$\norm[s]{\cdot}=\norm{A^{s/2}\cdot}$, $s=1,2$.

\begin{thm}\label{thm:main}
  Let $\{X(t)\}_{0\leq t\leq T},\{X^n\}_{0\leq n\leq N}$, and
  $\{\Xbe^n\}_{0\leq n\leq N}$ be the solutions of \eqref{eq:sac},
  \eqref{eq:ss10}--\eqref{eq:ss1c}, and \eqref{eq:be},
  respectively. Assume further that $\EE\|X_0\|^{18}_1<\infty$,
  $\EE\|X_0\|_2^2<\infty$, and $\|A^{1/2}Q^{1/2}\|_\HS<\infty$. Then,
  for any $\ts\leq \ts_0$ with $2\ts_0\beta^2<1$, there exist $c,C>0$,
  depending on $T$, $\EE\|X_0\|^{18}_1$, $\EE\|X_0\|_2^2$,
  $\|A^{1/2}Q^{1/2}\|_\HS$, and $\ts_0$, such that
\begin{align}\label{eq:tildeconv}
\EE\sup_{0\leq n\leq N}\|X^n-X(t_n)\|^2&\leq c\ts,
\\
\label{eq:beconv}
\EE\sup_{0\leq n\leq N}\|\Xbe^n-X(t_n)\|^2&\leq C\ts.
\end{align}
\end{thm}

\begin{proof}
  The first result, \eqref{eq:tildeconv}, is a consequence of
  Theorem~\ref{thm:tildeerror}, Theorem~\ref{thm:splitsteperror}, and
  the triangle inequality. The second, \eqref{eq:beconv}, follows from
  \eqref{eq:tildeconv} and from Theorem~\ref{thm:BESSerr} by using,
  again, the triangle inequality.
\end{proof}

Strong convergence results for time discretisation schemes for SPDEs
with globally Lipschitz coefficients, or at least some sort of linear
growth condition as in \cite{MR2393581}, are abundant, see, for
example, \cite{CoxVN,MR2073438,MR1964941,GyM,Haus1,Haus2} and the
references therein. In contrast, there are only few results on strong
convergence of time discretisation schemes for SPDEs with
superlinearly growing coefficients. Furthermore, almost all of these
results are establishing the fact of strong convergence with no rate
given \cite{MR3081484,GyM,GySS,KLLAC,Kur}. Only in a very recent paper
\cite{JP} do the authors obtain strong rates, by employing an
exponential integrator scheme for a class of SPDEs without a linear
growth condition on the coefficients. The methods used in the present
paper and the split-step scheme \eqref{eq:ss10}--\eqref{eq:ss1c}
itself are completely different from the methods and the numerical
scheme in \cite{JP} and also give a result for the strong rate of
convergence of the classical fully implicit backward Euler
method. Finally, we also mention the thesis \cite{Kop} where, in
Chapter~5, a weak rate of convergence is obtained for an SPDE with
space-time white noise in one spatial dimension and with a
non-globally Lipschitz (polynomial) semilinear term.

The paper is organised as follows. In Section~\ref{sec:prelim} we
gather some necessary background material on deterministic and
stochastic evolution equations. In Section~\ref{sec:reform} we rewrite
the original split-step scheme \eqref{eq:ss10}--\eqref{eq:ss1c} as an
explicit Euler discretisation of an auxiliary stochastic differential
equation \eqref{eq:pspde} involving bounded operators and globally
Lipschitz non-linearities.  In Lemmata~\ref{lem:qarkbnd}--\ref{lem:fprop}
we establish some key properties of the coefficients appearing in
\eqref{eq:pspde}. In Section~\ref{sec:momb} we first state our
hypothesis on the smoothness of the initial data and the regularity of
the covariance operator of $W$. In Section~\ref{subsec:bxkt} we prove
moment bounds on the solution of the auxiliary equation
\eqref{eq:pspde} making fully use of the boundedness of the
coefficients so that an It\^o formula can be rigorously employed. In
Subsection~\ref{subsec:dmomb} we establish various moment bounds on
the solution of the split-step scheme \eqref{eq:ss10}--\eqref{eq:ss1c}
based on energy arguments.  In Section~\ref{subsec:baux} we introduce
and compare a piecewise constant and a piecewise linear
continuous-in-time adapted process, both of which coincide with the
solution of the split-step scheme on the temporal grid. These
auxiliary processes turn out to be useful tools and they already
appear in the SDE setting, for example, in \cite{HMS}. In
Section~\ref{sec:error} we analyse the error in the split-step scheme
and in the backward Euler scheme and the conclusion of these results
are summarised in Theorem~\ref{thm:main} above.

\section{Preliminaries}\label{sec:prelim}

Throughout the paper we will use various norms for linear operators on
a Hilbert space. We denote by $\mathcal{L}(H)$ the space of bounded
linear operators on $H$ with the usual operator norm denoted by
$\norm{\cdot}$. If for a self-adjoint, positive semidefinite operator
$T\colon H\to H$, the sum
\begin{equation*}
\Tr T:=\sum_{k=1}^\infty\langle Te_k,e_k\rangle<\infty
\end{equation*}
for an orthonormal basis (ONB) $\{e_k\}_{k\in \mathbb{N}}$ of $H$,
then we say that $T$ is trace class. In this case $\Tr T$, the trace
of $T$, is independent of the choice of the ONB. If for an operator
$T\colon H\to H$, the sum
\begin{equation*}
\|T\|_{\HS}^2:=\sum_{k=1}^\infty\|Te_k\|^2<\infty
\end{equation*}
for an ONB $\{e_k\}_{k\in \mathbb{N}}$ of $H$, then we say that $T$ is
Hilbert--Schmidt and call $\|T\|_{\HS}$ the Hilbert--Schmidt norm of
$T$.  The Hilbert--Schmidt norm of $T$ is independent of the choice of
the ONB. We have the following well-known properties of the trace and
Hilbert--Schmidt norms, see, for example, \cite[Appendix C]{DPZ},
\begin{align}
\label{eq:hs}
\|T\|&\le \|T\|_{\HS},\quad \|TS\|_{\HS}\le \|T\|_{\HS}\|S\|,
\quad\|ST\|_{\HS}\le \|S\|\,\|T\|_{\HS},
\\
\label{eq:tr}
\Tr Q&=\|Q^{\frac12}\|_{\HS}^2=\|T\|^2_{\HS}=\|T^*\|_{\HS}^2,
\quad\text{ if $Q=TT^*$.}
\end{align}

Next, we introduce fractional order spaces and norms.  It is well
known that our assumption that $A$ is the Dirichlet Laplacian in the
smooth spatial domain $\mathcal{D}$ implies the existence of a sequence
of non-decreasing positive real numbers $\{\lambda_k\}_{k\geq 1}$ and
an ONB $\{e_k\}_{k\geq 1}$ of $H$ such that $Ae_k = \lambda_k e_k$ and
$\lim_{k\to\infty} \lambda_k = \infty$.  We define the scalar product
and norm of order $s\in \IR$:
\begin{equation*}\label{eq:fp}
\langle v,w\rangle_s:=\sum_{k=1}^{\infty}\lambda_k^{s}\langle
v,e_k\rangle \langle w,e_k\rangle,
\quad
\|v\|_s^2=\sum_{k=1}^{\infty}\lambda_k^{s}\langle v,e_k\rangle^2.
\end{equation*}
For $s\ge0$ the fractional order space is defined by
$\dot{H}^s=D(A^{s/2}):=\{v\in H:\|v\|_s<\infty\}$ and for $s<0$ as the
closure of $H$ with respect to the $\norm[s]{\cdot}$-norm.  It is well-known
that $\dH^1=H^1_0(\Dom)$ and $\dH^2=H^2(\Dom)\cap H^1_0(\Dom)$.

We recall the fact that the semigroup $\{E(t)\}_{t\ge 0}$ generated by
$-A$ is analytic because $-A$ is self-adjoint and negative definite,
see, for example, \cite[Example 3.7.5]{MR2798103}. For such a
semigroup it follows from the Spectral Theorem that for $\alpha\ge 0$,
\begin{align} \label{eq:anal0}
\|A^{\alpha}E(t)v\|\le C_{\alpha}t^{-\alpha}\|v\|,\quad t>0;
\quad C_\alpha=\sup\{\lambda^\alpha e^{-\lambda}:\lambda\ge 0\}. 
\end{align}
We will also use the Burkholder--Davies--Gundy inequality for
It\^o-integrals of the form
$\int_0^t\langle \eta(s),\dd \tilde{W}(s)\rangle$, where $\tilde{W}$
is a $\tilde{Q}$-Wiener process. For this kind of integral,
the Burkholder--Davies--Gundy inequality, \cite[Lemma 7.2]{DPZ}, takes
the form
\begin{equation}\label{eq:rbdg}
\EE\sup_{t\in [0,T]}
\Big|\int_0^t\langle \eta(s),\dd \tilde{W}(s)\rangle\Big|^p
\le C_p \EE\Big(\int_0^{T}\|\tilde{Q}^{\frac12}\eta(s)\|^2\,\dd s\Big)^{\frac{p}{2}},
~p\ge 2.
\end{equation}

Also, if $Y$ is an $H$-valued Gaussian random variable with
covariance operator $\tilde{Q}$, then, by \cite[Corollary 2.17]{DPZ},
we can bound its $p$-th moments via its covariance operator as
\begin{equation}\label{eq:es1a}
\EE\|Y\|^{2p}\le C_p (\mathbb{E}\|Y\|^2)^{p}
=C_p (\Tr \tilde{Q})^{p}=\|\tilde{Q}^{\frac12}\|_{\HS}^{2p}.
\end{equation}
We will repeatedly apply this to the It\^o integral $\int_s^tR\,\dd
W(r)$, where $R$ is a constant, possibly unbounded, operator on $H$ and
$W$ is a $Q$-Wiener process. Then
\eqref{eq:es1a} reads
\begin{equation}\label{eq:es1b}
\EE\Big\|\int_s^tR\,\dd W(r)\Big\|^{2p}\leq C_p(t-s)^p\|RQ^{1/2}\|_{\HS}^{2p}.
\end{equation}
If $p=1$,  the inequality in \eqref{eq:es1b} becomes an equality with $C=1$.
Moreover, the inequality
\begin{equation}\label{eq:psum}
\Big|\sum_{j=M}^Ka_j\Big|^p\leq |M-K+1|^{p-1}\sum_{j=M}^K|a_j|^p
\end{equation}
will be frequently utilised. This is a direct consequence of
H\"older's inequality.

The existence of solutions to \eqref{eq:sac} and their regularity has
been studied in, for example, \cite{Liu} (variational solution).  
We summarise some known results in the following theorem and we also refer to
the discussion preceding and after \cite[Proposition 3.1]{KLLAC} for
more details. To briefly describe the notion of a variational solution we consider
the mapping 
$
\tilde A: V\to V'
$
defined by
\begin{align*}
\langle\tilde Au,v\rangle 
=-\langle \nabla u, \nabla v \rangle -\langle f(u), v\rangle, \quad u,v\in V. 
\end{align*}
A continuous, $H$-valued, $\mathcal{F}_t$-adapted process
$X=\{X(t)\}_{t\in [0,T]}$ is called a \textit{variational solution} of
\eqref{eq:sac}, if for its $\dd t\otimes \mathbb{P}$ equivalence class
$\hat{X}$ we have
$\hat{X}\in L^\alpha([0,T]\times \Omega, \dd t\otimes \mathbb{P}; V
)$, for some $\alpha \ge 2$, and 
\begin{align*}
X(t)=X_0+\int_0^t \tilde A\hat X(s)\,\dd s+ W(t), \quad t\in [0,T], 
\ \text{$\mathbb{P}$-as.}  
\end{align*}

\begin{thm}\label{thm:existenceandregularityofX}
If $\|A^{\frac12}Q^{\frac12}\|_{\HS}<\infty$ and
$\EE\|X_0\|_1^p<\infty$ for some $p\ge 2$, then there is a
unique variational solution $X=\{X(t)\}_{t\in [0,T]}$ of \eqref{eq:sac}. Furthermore,
there is $C_T>0$ such that
\begin{equation}\label{eq:esup}
\EE\sup_{t\in [0,T]}\|X(t)\|^p+\EE\sup_{t\in [0,T]}\|X(t)\|_{1}^p\le C_T.
\end{equation}
In addition, $X$ is also a mild solution; that is, it satisfies, $\mathbb{P}$-almost surely, the equation
\begin{equation}\label{eq:mildsac}
X(t)=E(t)X_0-\int_0^tE(t-s)f(X(s))\,\dd s+\int_0^t E(t-s)\,\dd W(s),
\quad t\in[0,T].
\end{equation}
\end{thm}
Note that, as mentioned in the introduction, we have
$H^1_0(\mathcal D)\subset D_f$ and hence, by \eqref{eq:esup}, the term
$f(X(s))$ in \eqref{eq:mildsac} is well defined.  We shall also
utilise the following versions of Gronwall's lemma in continuous and
discrete time. For a proof of the former, see \cite{Stig} and for the
latter, see \cite{MR2744841}.

\begin{lem}[Generalised Gronwall lemma]\label{lem:gronwallCont}
Let $u\in L_1([0,T],\R)$ be non-negative. If $\alpha, C_1,C_2> 0$ and
\begin{equation*}
u(t)\leq C_1+C_2\int_0^t(t-s)^{\alpha-1}u(s)\,\dd s,\quad 0< t\leq T,
\end{equation*}
then there exists
${C}={C}(T,C_2,\alpha)>0$ such that
\begin{equation*}
u(t)\leq C_1{C},\quad 0\leq t\leq T.
\end{equation*}
\end{lem}

\begin{lem}[Discrete Gronwall lemma]\label{lem:gronwallDiscrete}
  Let $\{a_j\}_{j=0}^{N}$ be a real non-negative sequence.  If $C_1,\
  C_2> 0$ and 
\begin{equation*}
a_n\leq C_1+C_2\sum_{j=0}^{n-1}a_j, \ 0\leq n\leq N,
\quad \text{then }
a_n\leq C_1\ee^{C_2n}, \quad 0\leq n\leq N.
\end{equation*}
\end{lem}

We will use the notation $P_j$ to denote a positive polynomial of
degree $j$, that is, $P_j(x)\geq 0$ whenever the argument $x\in\IR^N$ has
positive components.

\section{Reformulation of the problem}\label{sec:reform}

For sufficiently small $\ts_0$, to be made precise in
Lemma~\ref{lem:fprop}, equations \eqref{eq:ss1a}--\eqref{eq:ss1c} are
all uniquely solvable if $\ts<\ts_0$ and thus, for $j=0,1,\dots, N-1$,
\begin{align}
X^0&=X_0,\nonumber\\
Y^j&=\Rk X^j,  \quad \text{where $\Rk:=(I+\textstyle{\hts} A)^{-1}$,}
\label{eq:ss2a}\\
Z^j&= \Jk Y^j,  \quad\ \, \text{where $\Jk(x):=(I+\ts f)^{-1} x$,}
\label{eq:ss2b}\\
X^{j+1}&=\Rk Z^j+\Rk\Delta W^j. \label{eq:ss2c}
\end{align}
We insert \eqref{eq:ss2a} into \eqref{eq:ss2b} to get
$Z^j= \Jk \Rk X^j$.  We then substitute this and \eqref{eq:ss2a} into
\eqref{eq:ss1b}, to get $Z^j=\Rk X^j-\ts f(\Jk\Rk X^j)$.
If this is substituted into \eqref{eq:ss2c}, we arrive at
\begin{equation}\label{eq:onestep}
\begin{aligned}
X^{j+1}&=\Rk^2 X^j-\ts\Rk f(\Jk\Rk X^j)+\Rk\Delta W^j\\
&=X^{j}-\ts A(I+\tfrac{\ts}4A)\Rk^2\Xj-\ts\Rk f(\Jk\Rk X^j)+\Rk\Delta W^j,
\end{aligned}
\end{equation}
where we used the identity  $\Rk^2=I-\ts A(I+\frac\ts4A)\Rk^2$. With the definitions
\begin{align*}
  \begin{aligned}
M_k&=(I+\tfrac\ts4A),
\quad &\Ak&=AM_k\Rk^2, \\
\fk(\cdot)&=f(\Jk \cdot),
\quad &\Fk(\cdot)&=\Rk\fk(\Rk\cdot),
  \end{aligned}
\end{align*}
we can write \eqref{eq:onestep} in the form
\begin{equation*}
X^{j+1}=X^{j}-\ts \Ak X^{j}-\ts\Fk(X^j)+\Rk\Delta W^j.
\end{equation*}
This may be viewed as the fully explicit Euler scheme, or the
Euler--Maruyama scheme, for the equation
\begin{equation}\label{eq:pspde}
\dd\Xtk(t)+\Ak\Xtk(t)\,\dd t+\Fk(\Xtk(t))\,\dd t=\Rk\,\dd W(t).
\end{equation}

Note that the operator $-\Ak$ is self-adjoint and negative definite
and hence it is the generator of an analytic semigroup
$\Ek(t)=\ee^{-t\Ak}$, and \eqref{eq:anal0} holds with $A$ and $E$
replaced by $\Ak$ and $\Ek$, respectively. As we shall see, $\Ak$ is
also a bounded operator on $H$ and $\Fk$ is Lipschitz continuous and
thus \eqref{eq:pspde} admits unique strong, weak, mild, and
variational solutions, if also $\|\Rk
Q^{1/2}\|_{\HS}<\infty$.
Further, such solutions coincide, see \cite[Appendix F]{PR}. In
particular, \eqref{eq:pspde} may be written in its mild form,
\begin{equation}\label{eq:mildpert}
\Xtk(t)=\Ek(t)X_0-\int_0^t\Ek(t-s)\Fk(\Xtk(s))\,\dd s+\int_0^t\Ek(t-s)\Rk\,\dd W(s)
\end{equation}
or in its strong form
\begin{equation*}
\Xtk(t)=-\int_0^t\Ak\Xtk(s)+\Fk(\Xtk(s))\,\dd s+\int_0^t\Rk\,\dd W(s).
\end{equation*}

\subsection{Properties of the new operators}

The operators $\Rk$, $\Mk$, and $\Ak$ introduced above are
self-adjoint positive definite and they commute with $A$.  In fact,
$\Mk$ is coercive in the sense that
\begin{equation}\label{eq:Mkco}
\|\Mk x\|\geq C\|x\|, \quad x\in \dot{H}^2,
\end{equation}
and the operators $\Rk$, $\Mk\Rk$ satisfy
\begin{align}\label{eq:resol}
\|\Rk x\|&\leq \|x\|,\quad x\in H, \\
\label{eq:MkRk}
\tfrac12\|x\|&\leq\|\Mk\Rk x\|\leq\|x\|,\quad x\in H,
\end{align}
and hence, for $\Ak$,
\begin{equation}\label{eq:AAk}
\tfrac12\|A\Rk x\|\leq\|\Ak x\|\leq\| A\Rk x\|,\quad x\in H.
\end{equation}
The operator $A\Rk$ is (essentially) the Yosida approximation of $A$
with the bounds
\begin{equation}\label{eq:akb}
\|A\Rk x\|\leq \ts^{-1} C\|x\|, \quad \|A\Rk x\|\leq  C\|x\|_2,
\end{equation}
and the interpolated version
\begin{equation}\label{eq:arkgk0}
  \|A\Rk x\|\leq \ts^{-s/2} C\|x\|_{2-s}, \quad 0\leq s\leq 2.
\end{equation}
An immediate consequence is that
\begin{equation}\label{eq:arkg}
\|A^{\alpha/2}\Rk x\|\leq\ts^{-s/2} C\|x\|_{\alpha-s},  \quad 0\leq
s\leq 2,\ \alpha\in \IR.
\end{equation}
Note also that \eqref{eq:AAk} and \eqref{eq:akb} show that, indeed,
$\Ak\in \mathcal{L}(H)$. We will use \eqref{eq:arkg}, in particular,
with $\alpha=s=1$, that is,
\begin{equation}\label{eq:AhRk}
\|A^{1/2}\Rk x\|\leq\ts^{-1/2} C\|x\|.
\end{equation}
If $\ts_1,\ts_2\geq 0$ are two time-steps, then we have that
\begin{equation}\label{eq:Rkdiff}
R_{\frac{\ts_1}2}-R_{\frac{\ts_2}2}=\tfrac12(\ts_2-\ts_1)A
R_{\frac{\ts_1}2}R_{\frac{\ts_2}2}
\end{equation}
 and, assuming $\ts_1\geq \ts_2$,
\begin{equation}\label{eq:Rkdiffnorm}
\|(R_{\frac{\ts_1}2}-R_{\frac{\ts_2}2})x\|\leq
C\ts_1^{s/2}\|R_{\frac{\ts_2}{2}}x\|_{s}\leq
C\ts_1^{s/2}\|x\|_{s},\quad 0\leq s\leq 2.
\end{equation}

We note that if $A^{1/2}Q^{1/2}$ is Hilbert--Schmidt, then it is
bounded according to \eqref{eq:hs} and also $A^{1/2}\Rk Q^{1/2}$ is
Hilbert--Schmidt, hence bounded. More precisely, the following result
holds.

\begin{lem}\label{lem:qarkbnd}
For all $k\ge 0$ we have that  $\|Q^{1/2}A^{1/2}\Rk\|\le \|A^{1/2}Q^{1/2}\|_{\HS}$.
\end{lem}

\begin{proof}
By \eqref{eq:hs} and \eqref{eq:tr} we have that
\begin{equation*}
\begin{aligned}
\|Q^{1/2}A^{1/2}\Rk\|&\leq \|Q^{1/2}A^{1/2}\Rk\|_{\HS}=
\|(Q^{1/2}A^{1/2}\Rk)^*\|_{\HS}=\|\Rk A^{1/2} Q^{1/2}\|_{\HS}\\
&\leq \|\Rk\| \| A^{1/2} Q^{1/2}\|_{\HS}\leq \| A^{1/2}
Q^{1/2}\|_{\HS}.
\end{aligned}
\end{equation*}
The second equality is motivated by the self-adjointness of all
involved operators and the boundedness in $H$ of   $A^{1/2} Q^{1/2}$
and $\Rk$.
\end{proof}

We will make use of the following
two deterministic error estimates.

\begin{lem}\label{thm:semierr}
Let $E(t)$ and $\Ek(t)$ be the analytic semigroups generated by $-A$
and $-\Ak$, respectively. Then for $x\in \dot{H}^{s-r}$, with $0\leq
r\leq s\leq 2$, we have that
\begin{align}
\|(E(t)-\Ek(t)\Rk)x\|&\leq C k^{\frac{s}{2}}t^{-\frac r2}\|x\|_{s-r},\quad t>0,
\label{eq:semierr2}\\
\|(E(t)-\Ek(t))x\|&\leq C k^{\frac{s}2}\|x\|_{s},\quad t>0.
\label{eq:semierr1}
\end{align}
\end{lem}

\begin{proof}
Note first that in order to show \eqref{eq:semierr2}, we have to prove that
\begin{equation}\label{eq:ref}
\|(E(t)-\Ek(t)\Rk)A^{\frac{r-s}2}\|\leq Ck^{\frac{s}2}t^{-\frac r2}, \quad t>0.
\end{equation}
Consider the function
\begin{equation*}
F(t,\lambda):=\ee^{-\lambda t}
-\frac{1}{1+\frac{k\lambda}{2}}\,\ee^{-\lambda t \frac{1+\frac{k\lambda}{4}}{(1+\frac{k\lambda}{2})^2}}.
\end{equation*}
In view of the spectral calculus for $A$ and the definitions of
$\Ek(t)$ and $\Rk$ using the notation $\sigma(A)$ for the spectrum of
$A$, we have that
\begin{align*}
\|(E(t)-\Ek(t)\Rk)A^{\frac{r-s}2}\|
= \sup_{\lambda \in \sigma(A)}|F(t,\lambda)| \lambda^{\frac{r-s}{2}}
\leq \sup_{\lambda >0}|F(t,\lambda)| \lambda^{\frac{r-s}{2}}.  
\end{align*}
Therefore, in order to prove \eqref{eq:ref} and hence \eqref{eq:semierr2}, we have to establish that
\begin{equation}\label{eq:Ftl}
|F(t,\lambda)|
\leq C {k^{\frac{s}{2}}}{t^{-\frac{r}{2}}} \lambda^{\frac{s-r}{2}}, \quad t,\lambda >0.
\end{equation}
We first manipulate $F$ as follows:
\begin{align*}
F(t,\lambda)
&=\ee^{-\lambda t}-\frac{1}{1+\frac{k\lambda}{2}}\,\ee^{-\lambda t \frac{1+\frac{k\lambda}{4}}{(1+\frac{k\lambda}{2})^2}}
=\ee^{-\lambda t}-\ee^{-\ln(1+\frac{k\lambda}{2})-t\lambda(1-k\lambda g(k\lambda))}
\\ &
=\ee^{-\lambda t}-\ee^{-\lambda t\big(1-k\lambda g(k\lambda)+\frac{\ln(1+\frac{k\lambda}{2})}{\lambda t}\big)}
\end{align*}
with $g(x)=\frac{3+x}{4(1+\frac{x}{2})^2}$. 
We set $G(x)=xg(x)$ and
$f(x)=G(x)-\frac{\ln(1+\frac{x}{2})}{\lambda t}$. An elementary
calculation shows that $G'(x)>0$ for $x \ge 0$, $f(0)=0$, and
$f'(x)\ge 0$ if $\lambda t\ge x\ge 2$.

Let first $\lambda t\ge \lambda k\ge 2$. Then
\begin{align*}
|F(t,\lambda)|
&
=\left| \lambda t \ee^{-\lambda t}\int_{f(k\lambda)}^{0}\ee^{\lambda t  x}\,\dd x\right|
\le \lambda t \ee^{-\lambda t}f(k\lambda)\ee^{\lambda t f(k\lambda)}
\le \lambda t \ee^{-\lambda t} G(k\lambda)\ee^{\lambda tG(k\lambda)}\frac{1}{1+\frac{k\lambda}{2}}
\\ &
=\ee^{-\lambda t(1-k\lambda g(k\lambda))}(\lambda t)^{1+\frac{r}{2}}
{k^{\frac{s}{2}}}{t^{-\frac{r}{2}}}(k\lambda)^{1-\frac{s}{2}}g(k\lambda)
\frac{\lambda^{\frac{s-r}{2}}}{1+\frac{k\lambda}{2}}.
\end{align*}
The function $h(x)=\ee^{-x\epsilon}x^{1+\alpha}$ is maximised at
$x=\frac{1+\alpha}{\epsilon}$ and
$h(\frac{1+\alpha}{\epsilon})=C_{\alpha}\epsilon^{-(1+\alpha)}$ with
$C_{\alpha}$ uniformly bounded in $\alpha$ on bounded subintervals of
$[0,\infty)$. If we set
$\epsilon=(1-k\lambda
g(k\lambda))=\frac{1+\frac{k\lambda}{4}}{(1+\frac{k\lambda}{2})^2}$
and $\alpha=\frac{r}{2}$, then, for $k\lambda \ge 2$ and
$0\le r\le s \le 2$,
\begin{align*}
|F(t,\lambda)|
&
\le C
  \frac{(1+\frac{k\lambda}{2})^{2+r}}{(1+\frac{k\lambda}{4})^{1+\frac{r}{2}}}
  \frac{1}{1+\frac{k\lambda}{2}}
  \frac{k^{\frac{s}{2}}}{t^{\frac{r}{2}}}(k\lambda)^{1-\frac{s}{2}}g(k\lambda)\lambda^{\frac{s-r}{2}}
\\ &
\le
C (k\lambda)^{\frac{r}{2}}\frac{k^{\frac{s}{2}}}{t^{\frac{r}{2}}}
(k\lambda)^{1-\frac{s}{2}}g(k\lambda)\lambda^{\frac{s-r}{2}}
\le C \frac{k^{\frac{s}{2}}}{t^{\frac{r}{2}}} \lambda^{\frac{s-r}{2}}.
\end{align*}
In the last inequality we used the fact that the function
$H(x)=x^{\beta}g(x)$ is bounded if $0\le \beta\le 1$.

Let now $\lambda k \ge 2$ and $\lambda t<\lambda k$; that is,
$1<\frac{k}{t}$. Then, for $0\le r\le s \le 2$,
\begin{align*}
|F(t,\lambda)|\le 2
\le 2 (\lambda  k)^{\frac{s-r}{2}}\left(\frac{k}{t}\right)^{\frac{r}{2}}
= 2{k^{\frac{s}{2}}}{t^{-\frac{r}{2}}} \lambda^{\frac{s-r}{2}}.
\end{align*}
Finally, let $\lambda k<2$. Then, as $G$ is monotone, we have
$G(x)=xg(x)\le G(2)=\frac{5}{8}<1$ for $x\in [0,2]$. We compute
\begin{align*}
F(t,\lambda)&
=\ee^{-\lambda t}-\frac{1}{1+\frac{k\lambda}{2}}\ee^{-\lambda t}
+\frac{1}{1+\frac{k\lambda}{2}}\,\ee^{-\lambda t}(1-\ee^{\lambda t  k\lambda g(k\lambda)})
\\ &
=\frac{\lambda k}{2}\frac{1}{1+\frac{k\lambda}{2}}\ee^{-\lambda t}
+\frac{1}{1+\frac{k\lambda}{2}}\,\ee^{-\lambda t}(1-\ee^{\lambda t k\lambda g(k\lambda)})
=:F_1(t,\lambda)+F_2(t,\lambda).
\end{align*}
For $0\le r\le s\le 2$, it follows that
\begin{equation*}
F_1(t,\lambda)
=\frac{\frac12 k^{\frac{s}{2}}(\lambda t)^{\frac{r}{2}}
\ee^{-\lambda t }(\lambda k)^{1-\frac{s}{2}}\lambda^{\frac{s-r}{2}}}
{t^{\frac{r}{2}}(1+\frac{k\lambda}{2})}
\le C \frac{k^{\frac{s}{2}}}{t^{\frac{r}{2}}} \lambda^{\frac{s-r}{2}}.
\end{equation*}
For $F_2$ we write
\begin{equation*}
F_2(t,\lambda)=\frac{1}{1+\frac{k\lambda}{2}}\ee^{-t\lambda}t\lambda\int_{k\lambda
  g(k\lambda)}^0\ee^{t\lambda x}\,\dd x,
\end{equation*}
and hence, for $\lambda k<2$ and $0\le r\le s\le 2$,
\begin{align*}
|F_2(t,\lambda)|
&\leq \big(\ee^{-t\lambda}t\lambda \big)
\big(k\lambda g(k\lambda)\big)\big(\ee^{\frac5{8}t\lambda}\big)
\\ &
=(t\lambda)^{1+\frac  r2}\ee^{-\frac3{8}t\lambda}
\frac{k^{\frac s2}}{t^{\frac r2}}(k\lambda)^{1-\frac
  s2}g(k\lambda)\lambda^{\frac{s-r}2}
\leq C\frac{k^{\frac s2}}{t^{\frac r2}}\lambda^{\frac{s-r}2}.
\end{align*}
This finishes the proof of \eqref{eq:Ftl} and hence of \eqref{eq:semierr2}.

For \eqref{eq:semierr1} we have, by \eqref{eq:semierr2} with $r=0$ and
\eqref{eq:Rkdiffnorm} with $k_2=0$, that
\begin{align*}
\|(E(t)-E_k(t))x\|
\le \|(E(t)-E_{k}(t)R_{\frac{k}{2}})x\|
+\|E_k(t)\|\|(R_{\frac{k}{2}}x-x)\|\le Ck^{\frac{s}{2}}\|x\|_s.
\end{align*}
This completes the proof. \end{proof}

We now turn to the non-linear operator $f\colon D_f\subset H \rightarrow H$ and
its approximations, where $f$ is defined as $f(u)(s)=P(u(s))$, $s\in
\mathcal D$, with $P(r)=r^3-\beta^2r$, $r\in \mathbb{R}$.  As noted
already, since $P$ is a third degree polynomial and the 
spatial dimension $d\leq 3$, we have, by Sobolev's inequality, that
$\dH^1\subset L_6(\cD)\subset D_f$.  We shall see  in
Lemma~\ref{lem:fprop} below that if 
$2k\beta^2<1$, then the equation
\begin{equation}\label{eq:jkexists}
z+kf(z)=x, \quad x\in H,
\end{equation}
has a unique solution $z\in H$ and hence we may define an inverse operator $\Jk:H\to H$ as
\begin{equation}\label{eq:Jk}
\Jk(x):=(I+kf)^{-1}(x).
\end{equation}
This is the resolvent operator of $f$. We will also establish that the range of $\Jk$ is contained
in $L_6(\mathcal{D})$ and hence in $D_f$ and therefore the Yosida approximation of $f$, that is,
\begin{equation}\label{eq:fk}
\fk(x):= f(\Jk(x)),\quad x\in H,
\end{equation}
is also well defined $H\to H$. Furthermore, since the resolvent operator $\Rk$ of $A$
is a bounded operator on $H$, it follows that
\begin{equation}\label{eq:Fk}
\Fk(x):=\Rk\fk(\Rk x)
\end{equation}
is well defined, too.

We shall now prove the above claims and also some further
important properties of the non-linear operators that we need in the sequel.

\begin{lem} \label{lem:fprop} Let $f$ be as described above and let
  $\Jk$, $\fk$, and $\Fk$ be as in \eqref{eq:Jk}, \eqref{eq:fk}, and
  \eqref{eq:Fk}, respectively. Also, assume that $\ts \leq \ts_0$ with
  $2\ts_0\beta^2<1$.
\begin{enumerate}
\item\label{it:jkexists} The equation \eqref{eq:jkexists} has a unique
  solution $z\in L_6(\cD)\subset \mathcal D_f \subset H$ for every $x\in H$ and thus $\Jk\colon H\to H$ is
  well-defined. Furthermore, $f(\Jk(x))=f(z)\in H$ so that $\fk\colon H\rightarrow H$
  is also well defined. If, in addition, $x\in \dH^1$, then also
  $z,f(z)\in \dH^1$.

\item \label{it:Jklip} The operator $\Jk\colon H\to H$ is Lipschitz
  continuous and obeys a linear growth condition on $\dH^1$ with
  constants that are bounded as $\ts\rightarrow 0$; more precisely,
\begin{align}\label{eq:Jklip}
\|\Jk(x)-\Jk(y)\|&\leq  C(\ts)\|x-y\|, \quad x,y\in H,\\
\label{eq:Jklip1}
\|\Jk(x)\|_1&\leq  C(\ts)\|x\|_1\quad x\in \dH^1,
\end{align}
with  $C(\ts)=\frac{1}{\sqrt{1-2\ts \beta^2}}$.

\item \label{it:oneside} The operators $\fk, \Fk\colon H\to H$ are
  Lipschitz continuous with Lipschitz constant $1/k$. Furthermore,
  they have the same one-sided Lipschitz and dissipativity
  properties as $f$;  that is, \eqref{eq:onesided1}, \eqref{eq:diss1},
  and \eqref{eq:diss2} hold true with $f$ replaced by $\fk$ or
  $\Fk$ and $\beta^2$ replaced by $\beta^2/(1-2\ts \beta^2)$.

\item \label{it:IJF} It holds that
\begin{equation}\label{eq:IJF}
x-\Jk(x)=\ts\fk(x), \quad x\in H.
\end{equation}

\item \label{it:FaFb} It also holds that
\begin{equation}\label{eq:FaFb}
\begin{split}
\langle F_{\alpha}(x)-F_{\beta}(y),x-y\rangle
&=
\langle
f_{\alpha}(R_{\frac\alpha 2}x)-f_{\beta}(R_{\frac\beta 2}y),
R_{\frac\alpha 2}x-R_{\frac\beta 2}y\rangle
\\ & \quad
+\langle
f_{\alpha}(R_{\frac\alpha 2}x), (R_{\frac\alpha 2}-R_{\frac\beta
  2})y\rangle
\\ & \quad
+\langle f_{\beta}(R_{\frac\beta 2}y), (R_{\frac\alpha
  2}-R_{\frac\beta 2})x\rangle,\quad x,y \in H.
\end{split}
\end{equation}

\end{enumerate}
\end{lem}

\begin{proof}
  To prove Item \eqref{it:jkexists} we consider the function
  $g(r)=r+kP(r)-s$ for arbitrary but fixed $s\in \R$. We have that
  $g'(r)=1+k3r^2-k\beta^2$, so if $k\beta^2<1$, then $g$ is monotone
  and since it is a cubic polynomial the equation $g(r)=0$ has exactly
  one solution $r=r(s)$. It follows that if $x$ is a continuous
  function on $\cD$, then equation \eqref{eq:jkexists} has a unique
  solution $z$ that is also continuous on $\mathcal{D}$, since $P$ is continuous
  function on $\R$.

  If $x\in H=L_2(\Dom)$ is arbitrary, then we may approximate it by continuous
  functions. Let $\{x_n\}_{n=1}^\infty$ be a sequence of continuous
  functions such that $\lim_{n\rightarrow \infty } x_n=x$ in $H$ and let
  $\{z_n\}_{n=1}^\infty$ be the corresponding solutions to
  \eqref{eq:jkexists} with $x$ replaced by $x_n$. Using the particular
  form $P(r)=r^3-\beta^2r$, we shall show that $z_n\rightarrow z$ in
  $L_6(\Dom)$ and that this implies $z^3_n\rightarrow z^3$ in
  $H$, where $z_n^3(s):=(z_n(s))^3$, $s\in \mathcal D$. This means that $f(z_n)\rightarrow f(z)$ in $H$,
  which in turn implies the solvability of \eqref{eq:jkexists} as well
  as the fact that $f_k$ is a well defined operator on $H$.

  To do this we note that
  $\langle f(x),x\rangle=\|x\|_{L_4(\Dom)}^4-\beta^2\|x\|^2$ and
  $\|f(x)\|^2=\|x\|_{L_6(\Dom)}^6-2\beta^2\|x\|^4_{L_4(\Dom)}+\beta^4\|x\|^2$.
  Taking the squared norm of both sides of \eqref{eq:jkexists} we
  therefore see that, if $2k\beta^2<1$, then
\begin{equation}\label{eq:L6bound}
\begin{aligned}
\|x_n\|^2
&
= \|z_n\|^2
+ 2k\langle f(z_n),z_n\rangle
+ k^2\|f(z_n)\|^2
\\ &
=\|z_n\|^2
+2 k\big(
\|z_n\|_{L_4(\Dom)}^4 -\beta^2\|z_n\|^2\big)
\\ &
\quad +k^2\big( \|z_n\|_{L_6(\Dom)}^6
-2\beta^2\|z_n\|^4_{L_4(\Dom)}
+\beta^4\|z_n\|^2\big)
\\ &
\geq  (1-2k\beta^2) \big(\|z_n\|^2+k^2 \|z_n\|_{L_6 (\Dom )}^6\big)
\\ &
\geq  (1-2k_0\beta^2)\big(\|z_n\|^2+k^2 \|z_n\|_{L_6 (\Dom )}^6\big).
\end{aligned}
\end{equation}
Since the sequence $\{x_n\}_{n=1}^\infty$ is uniformly
bounded in $H$ it follows that $\{z_n\}_{n=1}^\infty$ is uniformly
bounded in $L_6(\Dom)$.

To proceed we claim that $\{z_n^3\}_{n=1}^{\infty}$ is a Cauchy
sequence in $H$. To see this, take the squared norm of both
sides of the identity
\begin{equation}\label{eq:diffJk}
z_i-z_j+k(f(z_i)-f(z_j))=x_i-x_j
\end{equation}
and use that
\begin{equation*}
\|f(z_i)-f(z_j)\|^2=\|z^3_i-z^3_j\|^2-2\beta^2 \langle
f(z_i)-f(z_j),z_i-z_j\rangle+\beta^4\|z_i-z_j\|^2
\end{equation*}
to compute
\begin{equation}\label{eq:Jkcompute}
\begin{aligned}
\|x_i-x_j\|^2&=\|z_i-z_j\|^2+2k\langle
f(z_i)-f(z_j),z_i-z_j\rangle+\ts^2\|f(z_i)-f(z_j)\|^2\\
&=\|z_i-z_j\|^2+2k(1-k\beta^2)\langle
f(z_i)-f(z_j),z_i-z_j\rangle
\\ &\quad
+k^2\|z^3_i-z^3_j\|^2
+k^2\beta^4\|z_i-z_j\|^2.
\end{aligned}
\end{equation}
Thus, if $2k\beta^2<1$, then \eqref{eq:onesided1} implies that
\begin{equation}\label{eq:contlip}
\begin{aligned}
\|x_i-x_j\|^2&\geq \|z_i-z_j\|^2-2k(1-k\beta^2)
\beta^2\|z_i-z_j\|^2+k^2\|z^3_i-z^3_j\|^2\\
&\geq \|z_i-z_j\|^2-2k \beta^2\|z_i-z_j\|^2+k^2\|z^3_i-z^3_j\|^2\\
&\geq (1-2k \beta^2)\|z_i-z_j\|^2+k^2\|z^3_i-z^3_j\|^2.
\end{aligned}
\end{equation}
Thus,  $\{z_n^3\}_{n=1}^{\infty}$ is a Cauchy sequence in $H$.

It is not hard to see that if $x,y\in L_6(\Dom)$, then
\begin{equation}\label{eq:lsixconv}
\|x-y\|_{L_6(\Dom)}^6\leq C\|x^3-y^3\|^2
\end{equation}
for a sufficiently large positive number $C$. Indeed, note that
\eqref{eq:lsixconv} is equivalent to
\begin{equation*}
0\leq\int_{\Dom}\big(C(x^3-y^3)^2-(x-y)^{6}\big)\,\dd \xi
=\int_{\Dom}(x-y)^2\big(C(x^2+xy+y^2)^2-(x-y)^{4}\big)\,\dd \xi.
\end{equation*}
Thus, we need to find a $C$ such that
$P_C(t,s):=C(t^2+ts+t^2)^2-(t-s)^{4}\ge0$ for all real $t$ and $s$.
We have that
\begin{equation*}
P_C(t,s)=(C-1)(t^4+s^4)+2(C+2)(ts)(t^2+s^2)+3(C-2)(t^2s^2).
\end{equation*}
If $ts\geq 0$, then $C>2$ suffices. Assume $C>2$ and $ts<0$. Then
\begin{equation*}
P_C(t,s)\geq (C-1)(t^4+s^4)
-2(C+2)\big(\epsilon t^2s^2+\tfrac1{2\epsilon}(t^4+s^4)\big)
+3(C-2)t^2s^2.
\end{equation*}
Take $\epsilon=\frac{3(C-2)}{2(C+2)}$, so that
$\frac1{2\epsilon}=\frac{C+2}{3(C-2)}$. Then
\begin{equation*}
P_C(t,s)\geq \Big(C-1-\frac23\frac{(C+2)^2}{(C-2)}\Big)(t^4+s^4).
\end{equation*}
We must find $C>2$ such that
\begin{equation*}
C-1-\frac23\frac{(C+2)^2}{C-2}\ge0.
 \end{equation*}
Under the current restriction on $C$ this holds if
$C^2-17C-2\geq 0$.
This is true for large $C$ and \eqref{eq:lsixconv} is proved.

By combining \eqref{eq:contlip} and \eqref{eq:lsixconv}, we see that
$\{z_n^3\}_{n=1}^{\infty}$ is a Cauchy sequence in $L_6(\Dom)$. Thus,
$z_n\rightarrow z$ for some $z$ in $L_6(\Dom)$. It remains to show
that $z_n^3\rightarrow z^3$ in $H$. But
\begin{equation*}
\begin{aligned}
\|z^3_n-z^3\|^2&=\|(z_n^2+z_nz+z^2)(z_n-z)\|^2\leq
\|z_n^2+z_nz+z^2\|^2_{L_3(\Dom)}\|z_n-z\|_{L_{6}(\Dom)}^2\\
&\leq
C\big(\|z_n\|_{L_6(\Dom)}^4+\|z\|^4_{L_6(\Dom)}\big)\|z_n-z\|_{L_{6}(\Dom)}^2.
\end{aligned}
\end{equation*}
Since $z\in L_6(\Dom)$ and, by \eqref{eq:L6bound}, the sequence
$\{z_n\}_{n=1}^\infty$ is bounded in $L_6(\Dom)$ and
$z_n\rightarrow z$ in $ L_6(\Dom)$ the assertion is true. We have thus
shown that there is a solution $z \in L_6(\Dom)$ to
\eqref{eq:jkexists} for every $x\in H$.  Uniqueness follows from
\eqref{eq:contlip}. 

If $x\in \dH^1$, then take the squared $\dH^1$-norm of
\eqref{eq:jkexists} and use \eqref{eq:diss2} to get
\begin{equation}\label{eq:sqr1}
\begin{aligned}
\|x\|^2_1&=\|z\|^2_1+2k\langle f(z),z\rangle_1
+k^2\|f(z)\|^2_1\\
&\geq\|z\|^2_1-2 k\beta^2\|z\|^2_1+k^2\|f(z)\|^2_1.
\end{aligned}
\end{equation}
Since $2 k\beta^2<1$, we have that $z,f(z)\in \dH^1$ if $x\in
\dH^1$. We have proved Item~\eqref{it:jkexists}.

We turn to Item \eqref{it:Jklip}.  Lipschitz continuity of $J_k$ in
$H$ follows from \eqref{eq:contlip}, since
\begin{equation*}
\|\Jk(x_i)-\Jk(x_j)\|
=\|z_i-z_j\|\leq \frac1{\sqrt{1-2 k \beta^2}}\|x_i-x_j\|.
\end{equation*}
The linear growth bound
\eqref{eq:Jklip1} in $\dH^1$ follows from \eqref{eq:sqr1}.

To prove Item~\eqref{it:oneside}, we use the first equality in
\eqref{eq:Jkcompute} and \eqref{eq:onesided1} to conclude that if
$2k\beta^2<1$, then
\begin{equation*}
\|x_i-x_j\|\geq k\|f(z_i)-f(z_j)\|.
\end{equation*}
That is the claimed Lipschitz continuity of $f_k$ in $H$.
The claim about $F_k$ follows immediately from this and the
contraction property \eqref{eq:resol} of $\Rk$.

For the one-sided Lipschitz conditions, we use \eqref{eq:diffJk},
\eqref{eq:onesided1}, and the result in Item~\eqref{it:Jklip}. Indeed,
for $f_k$,
\begin{equation}\label{eq:fkl}
\begin{aligned}
&\langle f(z_1)-f(z_2),x_1-x_2\rangle =\langle
f(z_1)-f(z_2),z_1-z_2+k(f(z_1)-f(z_2)) \rangle\\
& \qquad\geq -\beta^2\|
z_1-z_2\|^2 +k\|f(z_1)-f(z_2)\|^2
\geq -\frac{\beta^2}{1-2\ts\beta^2}\|x_1-x_2\|^2.
\end{aligned}
\end{equation}
For $\Fk$, from \eqref{eq:fkl} and the the contraction property
\eqref{eq:resol} of $\Rk$, we have that
\begin{equation*}
\begin{aligned}
\langle \Fk(x_1)-\Fk(x_2),x_1-x_2\rangle
& =\langle\Rk ( \fk(\Rk x_1)-\fk(\Rk x_2)),x_1-x_2\rangle\\
&= \langle \fk(\Rk x_1)-\fk(\Rk x_2),\Rk x_1-\Rk x_2 \rangle\\
&\geq -\frac{\beta^2}{1-2k\beta^2}\|\Rk x_1-\Rk x_2\|
\geq -\frac{\beta^2}{1-2k\beta^2}\| x_1-x_2\|.
\end{aligned}
\end{equation*}
The fact that \eqref{eq:diss1} holds for $\fk$ and $\Fk$ follows from
the one-sided Lipschitz conditions and the fact that $\fk(0)=\Fk(0)=0$.

To see that \eqref{eq:diss2} holds for $\fk$ (with a different
constant) we use \eqref{eq:diss2} for $f$, the definitions of $z,f_k$,
and $J_k$ as well as the linear growth \eqref{eq:Jklip1} of $\Jk$ on
$\dH^1$ to conclude that
\begin{align*}
\langle \fk(x),x \rangle_1
&=\langle f(J_k(x)),z+kf(z) \rangle_1
=\langle f(z),z+kf(z) \rangle_1
=\langle f(z),z\rangle_1+k\|f(z)\|_1^2
\\ &
\geq-\beta^2\| z\|_1^2
= -\beta^2\| J_k(x)\|_1^2
\geq-\frac {\beta^2}{1-2\ts\beta^2}\|x\|_1^2.
\end{align*}
Therefore, as $\Rk$ is also a contraction on $\dH^1$, it follows that
\begin{align*}
\langle \Fk(x),x\rangle_1&=\langle \Rk\fk(\Rk x),x\rangle_1
=\langle \fk(\Rk x),\Rk x\rangle_1
\\ &
\geq -\frac {\beta^2}{1-2\ts\beta^2}\|\Rk x\|_1^2
\geq -\frac {\beta^2}{1-2\ts\beta^2}\| x\|_1^2.
\end{align*}
The statement in Item \eqref{it:IJF} is only a rearrangement of
\eqref{eq:jkexists} using  $z=\Jk(x)$.

To prove Item \eqref{it:FaFb} we first note that
\begin{equation*}
\begin{aligned}
\langle F_{\alpha}(x)-F_{\beta}(y),x-y\rangle&=\langle
F_{\alpha}(x),x-y\rangle-\langle F_{\beta}(y),x-y\rangle\\
&=\langle
f_{\alpha}(R_\alpha x),R_\alpha(x-y)\rangle-\langle f_{\beta}(R_\beta y),R_\beta(x-y)\rangle.
\end{aligned}
\end{equation*}
The statement follows by adding and subtracting the quantities $\langle
f_{\alpha}(R_\alpha x),R_\beta y\rangle$ and $\langle
f_{\beta}(R_\beta y),R_\alpha x\rangle$ and rearranging the terms.
\end{proof}

We end this section by making the useful observation that, in view of
\eqref{eq:IJF} and the fact that $\Rk x$ solves $(I+\tfrac k2A)y=x$,
we have
\begin{equation}\label{eq:yosidacomp}
x-\Jk (\Rk x)=x-\Rk x +\Rk x-\Jk(\Rk x)=\frac k2 A\Rk x+k\fk(\Rk x).
\end{equation}
Throughout the rest of the paper we will often write
$J_kx$ for $J_k(x)$ in order to increase readability of complicated
formulae even though $J_k$ is non-linear.

\section{Moment bounds on solutions}\label{sec:momb}
Various moment bounds on $\Xtk$ and $\Xj$ are crucial for our
convergence analysis.  In order to prove these we assume moment bounds
on the $\dH^1$-norm of the initial value and some spatial smoothness
of the noise.

\begin{assumption}\label{ass:q}\phantom{blah}\hfill
\begin{enumerate}
\item It holds that $k\le k_0$ with $2k_0\beta^2<1$.
\item For the covariance operator $Q$ of the Wiener process $W$, the
  inequality $\|A^{1/2}Q^{1/2}\|_{\HS}<\infty$ holds.
\item For some $q\ge1$, the $q^{\mathrm{th}}$ moment of the
  $\dH^1$-norm of the initial value $X_0$ is finite; that is, 
  $\EE\|X_0\|_1^q<\infty$.
\end{enumerate}
\end{assumption}

\subsection{Bounds on $\Xtk$.}\label{subsec:bxkt}
\begin{lem}\label{lem:tildebnd}
  Let $\Xtk(t)$ be the solution of \eqref{eq:pspde} and, for some
  $p\geq 1$, let Assumption~\ref{ass:q} hold with $q=2p$. Then there
  is $C_1=C_1(\ts_0,T,p,\HSAQ,\EE\|X_0\|_1^{2p})>0$ such that
\begin{equation*}
\EE\Big(\sup_{0\leq t\leq
  T}\|\Xtk(t)\|_1^{2p}\Big)\leq C_1
\end{equation*}
and $C_2=C_2(\ts_0,T,\HSAQ,\EE\|X_0\|_1^{2})>0$ such that
\begin{equation*}
\EE\int_0^T\|\Ak^{1/2}\Xtk(s)\|_1^2\,\dd s\leq C_2.
\end{equation*}
\end{lem}

\begin{proof}
Since the functional $\|\cdot\|^2_1$ is not continuous on $H$,
we can not apply It\^o's formula directly. But note, first formally,
that
\begin{equation}\label{eq:formal}
\begin{aligned}
\tY(t):&=A^{1/2}\Xtk(t)
=\Ek(t)A^{1/2}X_0-\int_0^t\Ek(t-s)A^{1/2}\Fk(\Xtk(s))\,\dd s
\\ & \quad
+\int_0^t\Ek(t-s)A^{1/2}\Rk\,\dd W(s)
\\ &
=\Ek(t)A^{1/2}X_0
-\int_0^t\Ek(t-s)A^{1/2}\Fk(A^{-1/2}\tY(s))\,\dd s
\\ & \quad
+\int_0^t\Ek(t-s)A^{1/2}\Rk\,\dd W(s).
\end{aligned}
\end{equation}
Thus $\tY(t)$ is the mild solution of
\begin{equation*}
\dd \tY+\big(\Ak\tY+A^{1/2}\Fk(A^{-1/2}\tY)\big)\,\dd
t=A^{1/2}\Rk\,\dd W,\ t\in(0,T);\ \tY(0)=A^{1/2}X_0.
\end{equation*}
Since $\Ak$ is bounded, $A^{1/2}\Fk(A^{-1/2}\cdot)$ is globally Lipschitz, and
$A^{1/2}\Rk Q^{1/2}$ is a Hilbert--Schmidt operator, the mild
solution exists and it equals the strong solution. Reverting the
computation in \eqref{eq:formal} one finds that $A^{-1/2}\tY(t)$
solves \eqref{eq:mildpert}. Hence, it does indeed hold that
$\tY(t)=A^{1/2}\Xtk(t)$.

Since $\tfrac12\norm{\cdot}^2$ is continuous on $H$, we have, almost
surely (\cite[Theorem~2.1]{Brzezniak}),
\begin{equation*}
\begin{aligned}
\tfrac12\|\tY(t)\|^2
&
=\tfrac12\|A^{1/2}X_0\|^2+\int_0^t\langle\tY(s),\dd \tY(s) \rangle
\\ & \quad
+\tfrac12\int_0^t\Tr\big(A^{1/2}\Rk Q^{1/2}(A^{1/2}\Rk Q^{1/2})^*\big)\,\dd s
\\ &
=\tfrac12\|A^{1/2}X_0\|^2-\int_0^t\langle\tY(s),
\Ak\tY(s)+A^{1/2}\Fk(A^{-1/2}\tY(s))\rangle\,\dd s
\\ & \quad
+ \int_0^t\langle\tY(s),A^{1/2}\Rk\,\dd W(s)\rangle
+\tfrac12\int_0^t\Tr\big(A^{1/2}\Rk Q^{1/2}(A^{1/2}\Rk Q^{1/2})^*\big)\,\dd s.
\end{aligned}
\end{equation*}
Now, $\langle\tY(s), \Ak\tY(s)\rangle=\|\Ak^{1/2}\tY(s)\|^2$ and
\begin{equation*}
\begin{aligned}
\langle
&
\tY(s), A^{1/2}\Fk(A^{-1/2}\tY(s))\rangle
=\langle A^{1/2}A^{-1/2}\tY(s), A^{1/2}\Fk(A^{-1/2}\tY(s))\rangle
\\ & \qquad
=\langle A^{-1/2}\tY(s), \Fk(A^{-1/2}\tY(s))\rangle_1
\geq -C(\ts)\|A^{-1/2}\tY(s)\|_1^2=-C(\ts)\|\tY(s)\|^2,
\end{aligned}
\end{equation*}
so that
\begin{equation}\label{eq:pathineq}
\begin{aligned}
\tfrac12\|\tY(t)\|^2+\int_0^t\|\Ak^{1/2}\tY(s)\|^2\,\dd s
&
\leq
\tfrac12\|A^{1/2}X_0\|^2+C(\ts)\int_0^t\|\tY(s)\|^2\,\dd s
\\ &\quad
+\int_0^t\langle\tY(s),A^{1/2}\Rk\,\dd W(s)\rangle
+t\|A^{1/2}\Rk Q^{1/2}\|_{\HS}^2.
\end{aligned}
\end{equation}
If we drop the integral on the left side and
raise the remaining parts to the power $p\geq 1$, take the supremum
in time and then the expectation, we find, with the aid of \eqref{eq:psum} and
H\"older's inequality, that
\begin{equation}\label{eq:that}
\begin{aligned}
\EE\sup_{0\leq t\leq T}\|\tY(t)\|^{2p}
&
\leq
C(\ts,p,T)\Big(\EE\|A^{1/2}X_0\|^{2p}
+\EE \int_0^T\|\tY(s)\|^{2p}\,\dd s 
\\ & \quad +\EE \sup_{0\leq t\leq T}\Big|\int_0^t\langle
\tY(s),A^{1/2}\Rk\,\dd W(s)\rangle \Big|^p+\|A^{1/2}\Rk
Q^{1/2}\|_{\HS}^{2p}\Big).
\end{aligned}
\end{equation}
Using first that $|ab|\leq \tfrac12(a^2+b^2)$ and then the
Burkholder--Davies--Gundy inequality \eqref{eq:rbdg}, we obtain
\begin{equation}
\begin{aligned}
\EE \sup_{0\leq t\leq T}\Big|\int_0^t\langle
\tY(s),A^{1/2}\Rk\,\dd W(s)\rangle\Big|^p
&\leq C\Big(1+\EE \sup_{0\leq t\leq T}\Big|\int_0^t\langle
\tY(s),A^{1/2}\Rk\,\dd W(s)\rangle\Big|^{2p}\Big)
\\ & 
\leq C(p)\Big(1+\EE\Big(\int_0^T \|Q^{1/2}A^{1/2}\Rk \tY(s)\|^2\,\dd s\Big)^{p} \Big).
\end{aligned}
\end{equation}
As, by Lemma~\ref{lem:qarkbnd}, we have that
\begin{align*}
\|Q^{1/2}A^{1/2}\Rk \tY(s)\|\leq\|Q^{1/2}A^{1/2}\Rk
\|\|\tY(s)\|\leq \HSAQ \,\|\tY(s)\|,
\end{align*}
we can use H\"older's inequality to obtain
\begin{equation}\label{eq:this}
  \begin{aligned}
\EE \sup_{0\leq t\leq T}\Big|\int_0^t\langle
\tY(s),A^{1/2}\Rk\,\dd W(s)\rangle\Big|^p
\leq C(p,T,\HSAQ)\Big( 1+
\EE\int_0^T\|\tY(s)\|^{2p}\,\dd s\Big).
  \end{aligned}
\end{equation}
Since
$\EE\int_0^T\|\tY(s)\|^{2p}\,\dd s=\int_0^T\EE\|\tY(s)\|^{2p}\,\dd
s\leq\int_0^T\EE\sup_{0\leq r\leq s}\|\tY(r)\|^{2p}\,\dd s $,
we deduce from \eqref{eq:this} and \eqref{eq:that} that
\begin{equation*}
  \begin{split}
\EE\sup_{0\leq t\leq T}\|\tY(t)\|^{2p}
\leq C(\ts_0,p,T,\HSAQ,\EE\|X_0\|_1^{2p})
\Big(1+\int_0^T\EE \sup_{0\leq r\leq s}\|\tY(r)\|^{2p}\,\dd s\Big).
  \end{split}
\end{equation*}
By Gronwall's lemma, there is $C_1=C_1(\ts_0,p,T,\HSAQ,
\EE\|X_0\|_1^{2p})$ such that
\begin{equation}\label{eq:result}
\EE\sup_{0\leq t\leq T}\|\tY(t)\|^{2p}\leq C_1.
\end{equation}
Returning to \eqref{eq:pathineq}, we can find
$C_2=C_2(k_0,T,\HSAQ,\EE\|X_0\|_1^{2})$ such that 
\begin{equation*}
\EE\int_0^t\|\Ak^{1/2}\tY(s)\|^2\,\dd s\leq C_2,
\end{equation*}
by using that the expectation of the It\^o integral vanishes and
\eqref{eq:result} with $p=1$.
\end{proof}

\begin{corollary}\label{cor:intF}
  Let $\Xtk(t)$ be the solution of \eqref{eq:pspde} and, for some
  $p\geq 1$, let Assumption~\ref{ass:q} hold with $q=6p$. Then there
  is $C=C(\ts_0,p,T,\HSAQ,\EE\|X_0\|_1^{6p})>0$ such that
\begin{equation*}
\EE\int_0^T\|\fk(\Rk\Xtk(s))\|^{2p}\,\dd s\leq C.
\end{equation*}
\end{corollary}

\begin{proof}
Since, by \eqref{eq:fgrowth} and \eqref{eq:Jklip1},
\begin{equation*}
\|\fk(\Rk\Xtk(s))\|\leq C (\|\Jk(\Rk\Xtk(s))\|_1^3+1)
\leq C(k) (\|\Xtk(s)\|_1^3+1),
\end{equation*}
it holds that
\begin{align*}
\EE\int_0^T\|\fk(\Rk\Xtk(s))\|^{2p}\,\dd s
&
\leq C(k,p)
\EE\int_0^T  \big( \|\Xtk(s)\|_1^{6p}+1\big)\,\dd s
\\ & 
\leq C(k,p,T)
\EE\sup_{0\leq s\leq T}\big(\|\Xtk(s)\|_1^{6p}+1\big)
\end{align*}
and the last expression is bounded by Lemma~\ref{lem:tildebnd}.
\end{proof}

\subsection{Moment bounds for the solution of the split-step
  scheme}\label{subsec:dmomb}

We shall now prove similar results for the solution of
\eqref{eq:ss10}--\eqref{eq:ss1c}. We note that under
Assumption~\ref{ass:q} with $q=2$, $X_0$ and $\Wj$ belong to
$\dH^1$ almost surely. Since $\Rk$ and $\Fk$ map $\dH^1$ into itself
the upcoming calculations are motivated almost surely.

To begin, take $\dH^1$-inner products with $\Yj$ in \eqref{eq:ss1a}
and $\Zj$ in \eqref{eq:ss1b} to get
\begin{align*}
\|\Yj\|_1^2+\tfrac \ts 2\|\Yj\|_2^2
&=\langle\Xj,\Yj\rangle_1
\leq \tfrac12\|\Xj\|_1^2+ \tfrac12\|\Yj\|_1^2,
\\
\|\Zj\|_1^2+\ts\langle f(\Zj),\Zj\rangle_1
&=\langle\Yj,\Zj \rangle_1
\leq \tfrac12\|\Yj\|_1^2+ \tfrac12\|\Zj\|_1^2.
\end{align*}
We invoke \eqref{eq:diss2}, assume that $k \leq k_0$ with
$2k_0\beta^2<1$, and rearrange to get
\begin{align*}
\|\Yj\|_1^2+\ts \|\Yj\|_2^2\leq \|\Xj\|_1^2,  \quad
\|\Zj\|_1^2\leq \Big(1+\frac{2\ts\beta^2}{ 1-2\ts\beta^2}\Big) \|\Yj\|_1^2.
\end{align*}
Thus,
\begin{equation}\label{eq:zxbound}
\|\Zj\|_1^2\leq \big(1+\ts C(\ts_0)\big)\|\Xj\|_1^2.
\end{equation}
Further, by taking the squared $\dH^1$-norm of \eqref{eq:ss1c}, we get
\begin{equation}\label{eq:xnorm}
\|X^{j+1}\|_1^2+k \|X^{j+1}\|_2^2+\tfrac {k^2}4 \|X^{j+1}\|_3^2
=\|\Zj\|_1^2+2\langle\Zj,\Wj\rangle_1+\|\Wj\|^2_1.
\end{equation}
By inserting the bound on $\|\Zj\|_1^2$ from \eqref{eq:zxbound} into
\eqref{eq:xnorm}, we get
\begin{equation*}
\begin{aligned}
\|X^{j+1}\|_1^2+k \|X^{j+1}\|_2^2\leq
\|\Xj\|_1^2+kC\|\Xj\|_1^2+2\langle\Zj,\Wj\rangle_1+\|\Wj\|^2_1,
\end{aligned}
\end{equation*}
where $C=C(k_0)$.  Summing up from $j=0$ to $j=N-1$, we arrive at
\begin{equation}\label{eq:omegaineq}
\|X^N\|_1^2+\sum_{j=1}^{N}k\|\Xj\|_2^2
\leq\|X_0\|_1^2
+ C \sum_{j=0}^{N-1}\Big(k\|\Xj\|_1^2
+\langle\Zj,\Wj\rangle_1+\|\Wj\|^2_1\Big).
\end{equation}
Neglecting, for the moment, the sum on the left side of
\eqref{eq:omegaineq}, raising the  remaining terms to the power $p\geq
1$, and recalling \eqref{eq:psum}, we have (with $C=C (p,\ts_0)$)
\begin{equation*}
\begin{aligned}
\|X^N\|_1^{2p}
&\leq C \bigg(\|X_0\|_1^{2p}
+\Big(\sum_{j=0}^{N-1}k\|\Xj\|_1^2 \Big)^p
\\ & \quad
+\Big|\sum_{j=0}^{N-1} \langle\Zj,\Wj\rangle_1\Big|^p
+\Big(\sum_{j=0}^{N-1} \|\Wj\|^2_1\Big)^p \bigg)
\\ &
\leq C \Big(\|X_0\|_1^{2p} +N^{p-1}k^{p-1}\sum_{j=0}^{N-1}k\|\Xj\|_1^{2p}
\\ & \quad
+\Big|\sum_{j=0}^{N-1} \langle\Zj,\Wj\rangle_1\Big|^{2p}
+N^{p-1}\sum_{j=0}^{N-1} \|\Wj\|^{2p}_1 \Big).
\end{aligned}
\end{equation*}
Here we take the supremum in time and then the expectation,  make
repeated use of \eqref{eq:psum}, and apply
\eqref{eq:es1a}, the Burkholder--Davis--Gundy inequality
\eqref{eq:rbdg} and Lemma~\ref{lem:qarkbnd} to
conclude that, with 
\begin{align*}
C=C(p,k_0,t_N,\EE\|X_0\|_1^{2p},\HSAQ)>0  
\end{align*}
changing from step to step,
\begin{equation}
\begin{aligned}
&\EE\sup_{0\leq j\leq N}\|\Xj\|_1^{2p}
\leq C\Big(1+t_N^{p-1} \sum_{l=0}^{N-1}k\EE
\sup_{0\leq j\leq l}\|\Xj\|_1^{2p}
\\ &\qquad
+\EE\sup_{0\leq n\leq N}\Big|\sum_{j=0}^{n-1}
\langle\Zj,\Wj\rangle_1\Big|^{2p}+N^{p-1}\sum_{j=0}^{N-1}
\EE \|\Wj\|^{2p}_1 \Big)
\\ &
\leq C \Big(1+
t_N^{p-1}\sum_{l=0}^{N-1}k\EE \sup_{0\leq j\leq l}\|\Xj\|_1^{2p}
\\ & \qquad
+\EE\Big(\sum_{j=0}^{N-1}
k\|A^{1/2}Q^{1/2}A^{1/2}\Zj\|^2\Big)^{p}
+N^{p-1}\sum_{j=0}^{N-1}
\big(\EE \|\Wj\|^{2}_1\big)^p \Big)
\\ &
\leq C
\Big(1+
t_N^{p-1}\sum_{l=0}^{N-1}k\EE \sup_{0\leq j\leq l}\|\Xj\|_1^{2p}
\\ & \qquad
+t_N^{p-1}\sum_{j=0}^{N-1}
\ts\EE\|A^{1/2}Q^{1/2}\|^{2p}\|\Zj\|_1^{2p}
+t_N^{p-1}\sum_{j=0}^{N-1}\ts \|A^{1/2}Q^{1/2}\|_{\HS}^{2p} \Big)
\\ &
\leq  C\Big(1+\sum_{l=0}^{N-1}k\EE \sup_{0\leq j\leq l}\|\Xj\|_1^{2p}\Big).
\end{aligned}
\end{equation}
Gronwall's lemma thus yields
$C=C(p,k_0,t_N,\EE\|X_0\|_1^{2p},\HSAQ)$ such that
\begin{equation*}
\EE\sup_{0\leq j\leq N}\|\Xj\|_1^{2p}\leq C.
\end{equation*}
Also, since $\|\Jk\Rk\Xj\|^2_1\leq C(\ts_0)
\|\Xj\|_1^2$ by \eqref{eq:Jklip1} and the contraction property of $\Rk$ on $\dH^1$,
it follows that
\begin{equation*}
\EE\sup_{0\leq j\leq N}\|\Jk\Rk\Xj\|_1^{2p}\leq C.
\end{equation*}
Furthermore, returning to \eqref{eq:omegaineq}, we easily show that,
if $\EE\|X_0\|_2^2<\infty$, then
\begin{equation*}
\EE \sum_{j=0}^{N}k\|\Xj\|_2^2 \leq
C(k_0,t_N,\EE\|X_0\|_1^{2},\EE\|X_0\|_2^{2},\HSAQ).
\end{equation*}

We summarise the above findings in the following lemma.

\begin{lem}\label{lem:discretebounds}
  If Assumption~\ref{ass:q} is satisfied with $q=2p$ for some
  $p\geq 1$, and $\{\Xj\}_{j=0}^{N}$ is the solution of
  \eqref{eq:ss10}--\eqref{eq:ss1c}, then
\begin{equation*}
\EE\sup_{0\leq j\leq N}\|\Xj\|^{2p}
+\EE\sup_{0\leq j\leq N}\|\Xj\|_1^{2p}
+\EE\sup_{0\leq j\leq N}\|\Jk\Rk\Xj\|_1^{2p}\leq C_1,
\end{equation*}
where $C_1=C_1(k_0,p,T,\HSAQ,\EE\|X_0\|_1^{2p})>0$.  If, in addition, $\EE\|X_0\|_2^2<\infty$, then also
\begin{equation*}
\EE \sum_{j=0}^{N}k\|\Xj\|_2^2 \leq C_2,
\end{equation*}
where $C_2=C_2(k_0,T,\HSAQ,\EE\|X_0\|_1^{2},\EE\|X_0\|_2^{2})>0$.
\end{lem}

\subsection{Auxiliary time interpolations}\label{subsec:baux}
For the error analysis in Section~\ref{sec:error} we introduce two
$\cF_t$-measurable time interpolations of $\Xj$:
the piecewise constant function
\begin{equation}\label{eq:xbar}
\Xb(t)=\Xj,\quad t\in[t_j,t_{j+1}),
\end{equation}
and the continuous stochastic interpolation
\begin{equation}\label{eq:xhat1}
  \begin{aligned}
    \Xh(t)=\Xj-(t-t_j)\Ak\Xj-(t-t_j)\Fk(\Xj)
+\Rk\big( W(t)- W(t_j)\big),\quad t\in[t_j,t_{j+1}).
  \end{aligned}
\end{equation}
Note that
\begin{equation*}
\Xh(t)=X_0-\int_0^t \big(\Ak\Xb(s)+\Fk(\Xb(s))\big)\,\dd s+\int_0^t\Rk \,\dd W(s)
\end{equation*}
and $\Xh(t_j)=\Xb(t_j)=\Xj$.

We will also need a certain bound on the stochastic interpolation,
$\Xh$. It is convenient to do this with the help of a result on the
difference between $\Xh$ and $\Xb$. As we will need a series of such
results below we begin by proving these.

\begin{prop}\label{prop:supEerr}
 If Assumption~\ref{ass:q} is satisfied with $q=6p$, then there exists
 \begin{align*}
C=C(k_0,p,T,\HSAQ,\EE\|X_0\|_1^{6p})>0
 \end{align*}
such that
\begin{equation*}
\sup_{0\leq t\leq T}\EE\|\Xh(t)-\Xb(t)\|^{2p}\leq C\ts^p.
\end{equation*}
\end{prop}

\begin{proof}
Note that if $t\in [t_j,t_{j+1})$, then, from the definitions of $\Xb$
and $\Xh$ in \eqref{eq:xbar} and \eqref{eq:xhat1}, respectively, we
have that
\begin{equation*}
\Xh(t)-\Xb(t)=-(t-t_j)\Ak\Xj-(t-t_j)\Fk(\Xj)+\int_{t_j}^t\Rk\,\dd  W(s).
\end{equation*}
Therefore,
\begin{equation}\label{eq:basest}
\begin{aligned}
\sup_{t_j\leq t\leq t_{j+1}}\EE\|\Xh(t)-\Xb(t)\|^{2p}&\leq
C(p)\Big(k^{2p}\EE\|\Ak\Xj\|^{2p} +k^{2p}\EE\|\Fk(\Xj)\|^{2p}\\
&\quad+\sup_{t_j\leq t\leq t_{j+1}}\EE\|\int_{t_j}^t\Rk\,\dd W(s)\|^{2p}\Big).
\end{aligned}
\end{equation}
By \eqref{eq:AAk} and \eqref{eq:AhRk}, we obtain
\begin{equation}\label{eq:triv}
\|\Ak\Xj\|\leq k^{-1/2}C\|\Xj\|_{1}.
\end{equation}
Moreover,
\begin{equation}\label{eq:Fkgrej}
\begin{aligned}
\|\Fk(\Xj)\|&=\|\Rk\fk(\Rk\Xj)\| \leq
k^{-1/2}C\|A^{-1/2}\fk(\Rk\Xj)\|\\
&\leq k^{-1/2} C (\ts_0)P_2(\|\Xj\|_1)\|\Xj\|\leq k^{-1/2} C(\ts_0)
P_3(\|\Xj\|_1),
\end{aligned}
\end{equation}
where we used, again, \eqref{eq:AhRk} and \eqref{eq:local}, utilising
also that $f(0)=0$, \eqref{eq:Jklip1} and finally the boundedness of $\Rk$.  For
the stochastic integral, by \eqref{eq:es1b},
\begin{equation}\label{eq:stocint}
\sup_{t_j\leq t\leq t_{j+1}}\EE\Big\|\int_{t_j}^t\Rk\,\dd W(s)\Big\|^{2p}\leq
C(p)\ts^{p}\|\Rk Q^{1/2}\|^{2p}_{\HS}.
\end{equation}
Thus, inserting the bounds in \eqref{eq:triv}, \eqref{eq:Fkgrej}, and
\eqref{eq:stocint} into \eqref{eq:basest}, we see that
\begin{equation*}
\begin{aligned}
\sup_{0\leq t\leq T}\EE\|\Xh(t)-\Xb(t)\|^{2p}&=\sup_{0\leq j\leq
  N-1}\sup_{t_j\leq t\leq t_{j+1}}\EE\|\Xh(t)-\Xb(t)\|^{2p}\\
&\leq
k^pC(\ts_0,p,\HSAQ)\Big(1+\sup_{0\leq j\leq N}\EE P_{6p}(\|\Xj\|_1)\Big)\\
&
\leq
\ts^pC(\ts_0,p,\HSAQ)\Big(1+\EE \sup_{0\leq j\leq N} P_{6p}(\|\Xj\|_1)\Big).
\end{aligned}
\end{equation*}
The claim now follows from Lemma~\ref{lem:discretebounds}.
\end{proof}

\begin{prop}\label{prop:Eoneerr}
  Assume that Assumption~\ref{ass:q} is satisfied with $q=6$ and
also  $\EE\|X_0\|_2^2<\infty$. Then there is
  $C=C(k_0,T,\HSAQ,\EE\|X_0\|_1^{6},\EE\|X_0\|_2^{2})>0$ such that
\begin{equation*}
\int_0^t\EE\|\Xh(s)-\Xb(s)\|_1^2\,\dd s\leq C\ts,\quad t\in [0,T].
\end{equation*}
\end{prop}

\begin{proof}
  As in the beginning of the previous proof we have, for
  $t\in [t_j,t_{j+1})$, that
\begin{equation*}
\begin{aligned}
\EE\|\Xh(t)-\Xb(t)\|_1^2&\leq \ts\Big(\ts\big(\EE\|\Ak\Xj\|_1^{2}
+\EE\|\Fk(\Xj)\|_1^{2}\big)+\|\Rk Q^{1/2}\|_{\HS}^{2}\Big)\\
&\leq \ts\Big(\EE\|\Xj\|_2^{2}
+\EE\|\fk(\Rk\Xj)\|^{2}+\|\Rk Q^{1/2}\|_{\HS}^{2}\Big),
\end{aligned}
\end{equation*}
as $\|\Ak\Xj\|_1\leq k^{-1/2}C\|\Xj\|_2$ and that
\begin{align*}
\|\Fk(\Xj)\|_1=\|A^{1/2}\Rk\fk(\Rk\Xj)\|\leq 
k^{-1/2}C\|\fk(\Rk\Xj)\|.
\end{align*}
Also, $\|\fk(\Rk\Xj)\|\leq C(\ts_0)P_3(\|\Xj\|_1)$ so if
$t\leq t_{n+1}$, then
\begin{equation*}
\begin{aligned}
&\int_0^t\EE\|\Xh(t)-\Xb(t)\|_1^2\leq
C(\ts_0)\ts\sum_{j=0}^{n}\ts\Big(\EE\|\Xj\|_2^2+\EE P_6(\|\Xj\|_1)+\|\Rk
Q^{1/2}\|_{\HS}^{2}\Big)\\
& \qquad
\leq \ts C(\ts_0,T,\HSAQ)\Big(1+\EE\sup_{0\leq j\leq n} \|\Xj\|_{1}^{6}+\sum_{j=0}^{n}k\EE\|\Xj\|_2^2\Big).
\end{aligned}
\end{equation*}
The claim now follows from Lemma~\ref{lem:discretebounds}.
\end{proof}

We will use the following result.

\begin{lem}\label{lem:justabove}
  If  Assumption~\ref{ass:q} is satisfied with
  $q=6p$, then there is a constant 
  \begin{align*}
  C=C(k_0,p,T,\HSAQ,\EE\|X_0\|_1^{6p})>0    
  \end{align*}
such that
\begin{equation*}
\int_0^t\EE\|\Xh(s)-\Xb(s)\|_1^{2p}\,\dd s\leq  C,\quad t\in [0,T].
\end{equation*}
\end{lem}

\begin{proof}
  Using that $\|\Ak\Xj\|_1\leq k^{-1}C\|\Xj\|_1$ and
  $\|\Fk(\Xj)\|\leq k^{-1}CP_3(\|Xj\|_1)$, it follows in a similar way
  as in the proof of Proposition~\ref{prop:supEerr} that if
  $t\leq t_{n+1}$, then
\begin{equation*}
 \int_0^t\EE\|\Xh(s)-\Xb(s)\|_1^{2p}\, \dd s\leq  C(\ts_0,p,T,\HSAQ)\Big(1+\EE\sup_{0\leq
   j\leq n} \|\Xj\|_1^{6p}\Big).
\end{equation*}
\end{proof}

\begin{corollary}\label{cor:jkrkxhat}
  If Assumption~\ref{ass:q} is satisfied with $q=6p$, then there is
  \begin{align*}
  C=C(k_0,p,T,\HSAQ,\EE\|X_0\|_1^{6p})>0    
  \end{align*}
such that
\begin{equation*}
\EE\int_0^t\|\Xh(s)\|_1^{2p}\,\dd s\leq C,\quad t\in [0,T].
\end{equation*}
\end{corollary}

\begin{proof}
As $\|\Xh(s)\|_1\leq\|\Xb(s)\|_1+\|\Xb(s)-\Xh(s)\|_1$,
the claim follows from Lemma~\ref{lem:discretebounds} and
Lemma~\ref{lem:justabove}.
\end{proof}

\section{Error bounds}\label{sec:error}
We shall analyse the error $e(t)$ by splitting it as
\begin{equation*}
e(t):=X(t)-\Xh(t)=\big(X(t)-\Xtk(t)\big)+\big(\Xtk(t)-\Xh(t)\big)=:e^1(t)+e^2(t).
\end{equation*}

\subsection{A bound for $e^1$}
We start comparing the solution of the perturbed problem
\eqref{eq:pspde} to the solution of the original problem
\eqref{eq:sac}.
The main line of argument is as follows.  In Lemma~\ref{lem:XaX} we
establish that $\Xtk$ converges in $L_2(\Omega,\cF,P;C([0,T],H))$ to a
mild solution of Allen--Cahn equation \eqref{eq:sac}, in particular, to
the unique variational solution.  We show, in the proof of
Theorem~\ref{thm:tildeerror}, that the sequence $\{\Xtk\}_{k>0}$
satisfies
\begin{equation*}
\EE\sup_{0\leq t\leq T}\|\Xt_{\ts_1}(t)-\Xt_{\ts_2}(t)\|^2\leq C \ts_2,
\quad \ts_1\le\ts_2,
\end{equation*}
for some constant $C$ depending on the initial data and $Q$. Thus,
when $\ts_1\to 0$ we get the desired bound for $e_1$.

The first step is to analyse the error in the stochastic convolution
terms. The proof is analogous to that of \cite[Theorem 2.1]{MR3273327}
and is based on the factorisation method of DaPrato and Zabczyk
\cite[Chapter 5]{DPZ} and the deterministic error estimates from
Lemma~\ref{thm:semierr}. We omit the details of the proof. We also
note that the rate of convergence is suboptimal in terms of the
regularity of the noise but a sharper (almost order $k$ instead of
$k^{1/2}$) estimate is not needed for our purposes and would require
and extended range for $s$ and $r$ in the deterministic error
estimates in Lemma~\ref{thm:semierr} (compare with
\cite[(2.2)]{MR3273327}).

\begin{lem}\label{lem:converrsup}
Let $W_A(t)=\int_0^t E(t-s)\,\dd W(s)$ and $W_{\Ak}(t)=\int_0^t
\Ek(t-s)\Rk\,\dd W(s)$ and suppose that $\|A^{1/2}Q^{1/2}\|_{\HS}^2<
\infty$ and  $2p\geq 1$. Then, with $C=C(T,p,\HSAQ)>0$,
\begin{equation}\label{eq:converrsup}
\EE\sup_{0\leq t\leq T}\|W_{A}(t)-W_{\Ak}(t)\|^{2p}\leq Ck^{p}.
\end{equation}
\end{lem}

Next we prove that $\Xtk$ converges uniformly strongly to $X$ but with
no specific rate given. The proof is in the same spirit as the proof
of \cite[Theorem 5.4]{KLM} (see also \cite{MR3273327}).
\begin{lem}\label{lem:XaX}
Let  $\Xtk$ be the solution of \eqref{eq:pspde} and $X$ of
\eqref{eq:sac}. If Assumption~\ref{ass:q} is satisfied with $q=4$, then
\begin{equation*}
\lim_{k\rightarrow 0}\EE\sup_{0\leq t\leq T}\|X(t)-\Xtk(t)\|^2=0.
\end{equation*}
\end{lem}

\begin{proof}
  It follows from Theorem~\ref{thm:existenceandregularityofX},
  Lemma~\ref{lem:tildebnd}, and Lemma~\ref{lem:converrsup} that there
  is $K_T>1$ such that
  \begin{align*}
  \EE\sup_{0\leq s\leq T}\left(\|X(s)\|_1^2+\|X(s)\|^4\right)\leq 
  K_T,\quad \EE\sup_{0\leq s\leq 
    T}\left(\|\Xtk(s)\|_1^2+\|\Xtk(s)\|^4\right)\leq K_T    
  \end{align*}
and
  $\EE\sup_{0\leq t\leq T}\|W_{A}(t)-W_{\Ak}(t)\|^{2}\leq K_T  k$.
  Therefore, by Chebychev's inequality, for every $0<\epsilon<1$
  and $0<k<k_0$, there is $\Omega_{\epsilon,k}\subset \Omega$ with
  $P(\Omega_{\varepsilon,k}^c)<\epsilon$ such that uniformly for all
  $\omega\in \Omega_{\varepsilon,k}$ and $t\in [0,T]$ we have
\begin{align}
    \label{eq:wwk1}
    &\max(\|X(t)\|_1^2+\|X(t)\|^4,\|\Xtk(t)\|_1^2+\|\Xtk(t)\|^4)
    \leq K_T\epsilon^{-1} ,\\
\label{eq:wwk}
&\|W_{A}(t)-W_{\Ak}(t)\|^{2}\leq K_T k\epsilon^{-1} .
\end{align}
Let $0<\epsilon<1$, $0<k<k_0$ and $\omega\in \Omega_{\varepsilon,k}$
and without the loss of generality suppose that also
$\|X_0\|_1\le K_T\epsilon^{-1}$ on
$\Omega_{\varepsilon,k}$. Using the mild formulations
\eqref{eq:mildsac} and \eqref{eq:mildpert} we see that
\begin{equation}\label{eq:omegaerr}
\begin{aligned}
\|X(t)-\Xtk(t)\|
&\leq \|(E(t)-\Ek(t))X_0\|
\\ &\qquad
+
\Big\|\int_0^t\big(E(t-s)f(X(s))-\Ek(t-s) \Fk(\Xtk(s))\big)\,\dd s\Big\|
\\&\qquad+\|\WA(t)-\Wk(t)\|.
\end{aligned}
\end{equation}
We have, immediately from \eqref{eq:semierr1} that
\begin{equation}\label{eq:init}
\sup_{0\leq t\leq
  T}\|(E(t)-\Ek(t))X_0\|\leq k^{1/2}C\|X_0\|_1\le CK_T\epsilon^{-1}k^{1/2}
\end{equation}
and from \eqref{eq:wwk}
\begin{equation}\label{eq:WAWAK}
\sup_{0\leq t\leq T}\|\WA(t)-\Wk(t)\|\leq K_T^{1/2}\epsilon^{-1/2}k^{1/2}.
\end{equation}
To analyse the second term we split the integrand as
\begin{equation}\label{eq:nonlinsplit}
\begin{aligned}
&E(t-s)\big(f(X(s))-f(\Xtk(s))\big)+
\big(E(t-s)-\Ek(t-s)\Rk)f(\Xtk(s)\big)
\\ &\qquad
+\Ek(t-s)\Rk\big(f(\Xtk(s))-f_k(\Rk\Xtk(s))\big)
=:e_1+e_2+e_3.
\end{aligned}
\end{equation}
Applying \eqref{eq:anal0} and \eqref{eq:local} yields
\begin{equation*}
\begin{aligned}
\|e_1\|&=\|A^{1/2}E(t-s)A^{-1/2}(f(X(s))-f(\Xtk(s)))\|\\
&\leq C
(t-s)^{-1/2}\|A^{-1/2}(f(X(s))-f(\Xtk(s)))\|\\
&\leq
(t-s)^{-1/2}(\|X(s)\|_1^2+\| \Xtk(s)\|_1^2+1)\|X(s)-\Xtk(s)\|,
\end{aligned}
\end{equation*}
so that
\begin{equation}\label{eq:e1}
\begin{aligned}
\Big\|\int_0^te_1\,\dd s\Big\|
&\leq C\sup_{0\leq s\leq T}(\|X(s)\|_1^2+\|\Xtk(s)\|_1^2+1)
\int_0^t
(t-s)^{-1/2} \|X(s)-\Xtk(s)\|\,\dd s
\\ & 
\leq C\big(1+2K_T\epsilon^{-1}\big) \int_0^t
(t-s)^{-1/2} \|X(s)-\Xtk(s)\|\,\dd s.
\end{aligned}
\end{equation}
By Lemma~\ref{thm:semierr} we have that
\begin{equation*}
\|e_2\|\leq C (t-s)^{-1/2}\|f(\Xtk(s))\|k^{1/2},
\end{equation*}
and from \eqref{eq:wwk1} it follows that
\begin{equation*}
\|f(\Xtk(s))\|\leq C\big( \|\Xtk(s)\|_1^3+1\big)
\leq C\big(1+K_T^{3/2}\epsilon^{-3/2}\big),\quad s\in [0,T].
\end{equation*}
Thus,
\begin{equation}\label{eq:e2}
\sup_{0\leq t\leq T}\Big\|\int_0^te_2(s)\,\dd s\Big\|
\leq C(1+K_T^{3/2}\epsilon^{-3/2})k^{1/2}.
\end{equation}
We continue with the observation that
\begin{equation*}
\begin{aligned}
e_3&=(A\Mk)^{1/2}\Rk\Ek(t-s)(A\Mk)^{-1/2}(f(\Xtk(s))-
f(\Jk\Rk\Xtk(s)))\\
&=\Ak^{1/2}\Ek(t-s)(A\Mk)^{-1/2}(f(\Rk\Xtk(s))-
f(\Xtk(s))).
\end{aligned}
\end{equation*}
Therefore, by \eqref{eq:anal0}, \eqref{eq:local}, and \eqref{eq:yosidacomp},
\begin{equation*}
\begin{aligned}
\|e_3\| &\leq (t-s)^{-1/2}C\|(A\Mk)^{-1/2}(f(\Rk\Xtk(s))-
f(\Xtk(s))) \|\\
& \leq (t-s) ^{-1/2}C\|A^{-1/2}(f(\Rk\Xtk(s))-
f(\Xtk(s))) \|\\
& \leq (t-s)^{-1/2} C (\|\Xtk(s)\|_1^2+1)\|\Xtk(s)-
\Jk\Rk\Xtk(s) \|\\
& \leq (t-s)^{-1/2} C (\|\Xtk(s)\|_1^2+1)k(\|f(\Jk\Rk\Xtk(s))\|+\|A\Rk\Xtk(s)\|).
\end{aligned}
\end{equation*}
As $\|A\Rk\Xtk(s)\|\leq C \ts^{-1/2}\|\Xtk(s)\|_1$ and
$\|f(\Jk\Rk\Xtk(s))\|\leq C(\|\Xtk(s)\|_1^3+1)$, we get in the same
way as above that
\begin{equation}\label{eq:e3}
\sup_{0\leq t\leq T}\Big\|\int_0^te_3(s)\,\dd s\Big\|
\leq C(1+K_T\epsilon^{-1})K_T^{1/2}\epsilon^{-1/2}C(1+K_T^{3/2}\epsilon^{-3/2}) k^{1/2}.
\end{equation}
Summarising the above findings, we have
\begin{align*}
\|X(t)-\Xtk(t)\|
\leq
C\big(1+K_T^3\epsilon^{-3}\big)k^{1/2}
+C\big(1+2K_T\epsilon^{-1}\big) \int_0^t(t-s)^{-1/2}\|X(s)-\Xtk(s)\|\,\dd s.
\end{align*}
It follows from Gronwall's lemma, Lemma~\ref{lem:gronwallCont}, that
there is  $C_\epsilon=C(T,K_T\epsilon^{-1})$ (growing rapidly as
$\epsilon\to 0$) such that uniformly on $\Omega_{\epsilon,k}$,
\begin{equation}\label{eq:asconv}
\sup_{0\leq t\leq T}\|X(t)-\Xtk(t)\|\leq C_\epsilon k^{1/2}.
\end{equation}
Despite the large constant in \eqref{eq:asconv}, it is now possible to
prove strong convergence of $\Xtk$ to $X$ but with no rate. Indeed,
\begin{equation*}
\begin{aligned}
&\EE\sup_{0\leq t\leq T}\|X(t)-\Xtk(t)\|^2\leq \int_{\Omega_{\epsilon,k}}\sup_{0\leq t\leq T}\|X(t)-\Xtk(t)\|^2\,\dd P\\
&\qquad\quad +2\int_{\Omega_{\epsilon,k}^c}\sup_{0\leq t\leq T}\Big(\|X(t)\|^2+\|\Xtk(t)\|^2\Big)\,\dd P\\
&\qquad\leq C_{\epsilon}^2k+4 \epsilon^{1/2}\Big(\int_{\Omega_{\epsilon,k}^c}\sup_{0\leq t\leq T}\Big(\|X(t)\|^4+\|\Xtk(t)\|^4\Big)\,\dd P\Big)^{1/2}\\
&\qquad\leq C_{\epsilon}^2k+4 \epsilon^{1/2}\Big(\EE\sup_{0\leq t\leq T}\Big(\|X(t)\|^4+\|\Xtk(t)\|^4\Big)\Big)^{1/2}
\\ &\qquad\leq C_{\epsilon}^2k+8 \epsilon^{1/2}K_T^{1/2}.
\end{aligned}
\end{equation*}
Let $\delta>0$ and choose $0<\epsilon<1$ such that
$8 \epsilon^{1/2}K_T^{1/2}<\delta/2$. Then, for
$k<\tfrac12 C_{\epsilon}^{-2}\delta$, we have that
$
\EE\sup_{0\leq t\leq T}\|X(t)-\Xtk(t)\|^2<\delta
$
and the proof is complete.
\end{proof}

Now, we are ready to prove the error estimate. The proof of the
fact that $\{\Xtk\}_{k>0}$ is a Cauchy sequence benefits from
techniques in \cite[Chapter 4]{MR2111320}.

\begin{thm}\label{thm:tildeerror}
  Let $\Xtk$ be the solution of \eqref{eq:pspde} and $X$ of
  \eqref{eq:sac}. If Assumption~\ref{ass:q} is satisfied with $q=6$,
  then there is $C=C(\ts_0, T,\HSAQ,\EE\|X_0\|^6_1)>0$ such that
\begin{equation}
\EE\sup_{0\leq t\leq T}\|\Xtk(t)-X(t)\|^2\leq C\ts.
\end{equation}
\end{thm}

We suppress the dependency of the constants $C$ on the data in
order to, we hope, increase the readability of the following proof.

\begin{proof}
We start by proving that $\{\Xtk\}_{k>0}$ is Cauchy
in $L_2(\Omega,\cF,P;C([0,T],H))$. To this
aim we pick two time-steps $\tsa$ and $\tsb$ and assume, without loss
of generality, that $\tsb\leq \tsa\leq \ts_0$. The corresponding solutions of
\eqref{eq:pspde} are $\Xta$ and $\Xtb$. We note that $\Xta(t)-\Xtb(t)$
solves the equation
\begin{equation*}
\begin{aligned}
&\dd \big(\Xta-\Xtb\big)
+\big(\Aa\Xta-\Ab\Xtb\big)\,\dd t
+\big(\Fa(\Xta)-\Fb(\Xtb)\big)\,\dd t
=(\Ra-\Rb)\,\dd W,\quad t\in (0,T], 
\\ &
\Xta(0)-\Xtb(0)=0.
\end{aligned}
\end{equation*}
Applying It\^o's formula to $\tfrac12\|\Xta(t)-\Xtb(t)\|^2$ thus gives
\begin{equation*}
\begin{aligned}
\tfrac12\|\Xta(t)-\Xtb(t)\|^2
&
=-\int^t_0 \langle \Xta(s)-\Xtb(s),\Aa\Xta(s)-\Ab\Xtb(s)\rangle\,\dd s
\\ &\quad
-\int_0^t\langle \Xta(s)-\Xtb(s),\Fa(\Xta(s))-\Fb(\Xtb(s))
\rangle\,\dd s
\\ &\quad
+\int_0^t\langle \Xta(s)-\Xtb(s),\big(\Ra-\Rb\big)\, \dd W(s)\rangle
\\ &\quad
+\tfrac12\int^t_0
\Tr\big(\big(\Ra-\Rb\big)Q^{1/2}
(\big(\Ra-\Rb\big)Q^{1/2})^*\big)\,\dd s
\\ &=:-I_1-I_2+I_3+I_4.
\end{aligned}
\end{equation*}
For $I_4$ we note that, by \eqref{eq:Rkdiff} and \eqref{eq:Rkdiffnorm},
\begin{equation*}
\begin{aligned}
&\Tr(\big(\Ra-\Rb\big)Q^{1/2}(\big(\Ra-\Rb\big)Q^{1/2})^*)
=\|\big(\Ra-\Rb\big)Q^{1/2}\|_{\HS}^2
\\ & \qquad
\leq C\tsa^2\|A\Ra\Rb Q^{1/2}\|_{\HS}^2
\leq C \tsa \|A^{1/2}\Rb Q^{1/2}\|_{\HS}^2.
\end{aligned}
\end{equation*}
Thus,
\begin{equation}\label{eq:I4}
\EE\sup_{0\leq t\leq T}|I_4(t)|\leq \tsa C T\|A^{1/2} Q^{1/2}\|_{\HS}^2.
\end{equation}
For $I_3$ we have, using again \eqref{eq:Rkdiff}, that
\begin{equation*}
\langle \Xta(s)-\Xtb(s),\big(\Ra-\Rb\big)  \,\dd W(s)\rangle
=\tfrac12(\tsa-\tsb)\langle \Xta(s)-\Xtb(s),A\Ra\Rb\, \dd W(s)\rangle.
\end{equation*}
Hence, using Cauchy's inequality, the Burkholder--Davies--Gundy
inequality \eqref{eq:rbdg}, H\"older's inequality, \eqref{eq:AhRk},
and Lemma~\ref{lem:qarkbnd}, we have
\begin{equation}\label{eq:I3}
\begin{aligned}
&
\EE\sup_{0\leq t\leq T}|I_3(t)|
\leq C\tsa \EE\sup_{0\leq t\leq T}\Big| \int_0^t\langle
\Xta(s)-\Xtb(s),A\Ra\Rb \,\dd W(s)\rangle\Big|\\
&\qquad \leq C\tsa\Big(1+ \EE\sup_{0\leq t\leq T}\Big| \int_0^t\langle
\Xta(s)-\Xtb(s),A\Ra\Rb \,\dd W(s)\rangle\Big|^2\Big)\\
& \qquad \leq C\tsa\Big(1+\EE\int_0^T \|Q^{1/2}A^{1/2}\Rb A^{1/2}\Ra
(\Xta(s)-\Xtb(s))\|^2\,\dd s \Big)\\
& \qquad \leq C\tsa\Big(1+\EE\int_0^T\|Q^{1/2}A^{1/2}\Rb\|^2\| A^{1/2}\Ra
(\Xta(s)-\Xtb(s))\|^2\,\dd s \Big)\\
& \qquad \leq C\tsa \Big(1+\frac 1{\tsa}\EE\int_0^T \|Q^{1/2}A^{1/2}\Rb\|^2\|\Xta(s)-\Xtb(s)\|^2\,\dd s \Big)\\
&\qquad \leq C\Big(\tsa +\int_0^T\EE\sup_{0\leq t\leq  s}\|\Xta(t)-\Xtb(t)\|^2\,\dd s\Big).
\end{aligned}
\end{equation}

To squeeze out smallness of  $I_2$ we first note that by
\eqref{eq:onesided1},
\begin{equation*}
\begin{aligned}
a_2^+&:=\langle f(\Ja \Ra \Xta(s))-f(\Jb \Rb \Xtb(s)),\Ja \Ra \Xta(s)-\Jb
\Rb \Xtb(s)\rangle\\
&\quad +\beta^2\| \Ja \Ra \Xta(s)-\Jb
\Rb \Xtb(s)\|^2\geq 0.
\end{aligned}
\end{equation*}
From \eqref{eq:IJF} it follows that
\begin{equation*} 
\begin{aligned}
\| \Ja \Ra \Xta(s)-\Jb\Rb \Xtb(s)\|^2
&\leq
C\tsa^2\Big(\|\fa(\Ra\Xta(s))\|^2+\|\fb(\Rb\Xtb(s))\|^2\Big)
\\ & \quad+C\|\Ra \Xta(s)-\Rb\Xtb(s)\|^2,
\end{aligned}
\end{equation*}
and from \eqref{eq:Rkdiff} that
\begin{equation*} 
\begin{aligned}
\|\Ra \Xta(s)-\Rb\Xtb(s)\|&\leq \|\Ra \Xta(s)-\Ra\Xtb(s)\|+\|\Ra
\Xtb(s)-\Rb\Xtb(s)\|\\
&\leq C \Big(\|\Xta(s)-\Xtb(s)\|+\tsa\|A\Ra\Rb\Xtb(s)\|\Big).
\end{aligned}
\end{equation*}
Combining the last three inequalities, inserting $f_\alpha=f(J_\alpha\cdot)$, we find that
\begin{equation}\label{eq:apf}
\begin{aligned}
0\leq a^+_2&\leq \langle \fa(\Ra \Xta(s))-\fb( \Rb \Xtb(s)),\Ja \Ra \Xta(s)-\Jb
\Rb \Xtb(s)\rangle\\
&\quad +C \Big(\|\Xta(s)-\Xtb(s)\|^2+\tsa^2\big(\|\fa(\Ra\Xta(s))\|^2\\
&\qquad+\|\fb(\Rb\Xtb(s))\|^2+\|A\Ra\Rb\Xtb(s)\|^2\big)\Big).
\end{aligned}
\end{equation}
 We shall subtract $I_2$ from $I_2^+:=\int_0^ta_2^+\,\dd s$ but first
 we note, using  \eqref{eq:FaFb} and \eqref{eq:Rkdiffnorm} as well as H\"older's and
 Cauchy's inequalities, that
\begin{equation}\label{eq:I2bound}
\begin{aligned}
-I_2&\leq\int_0^t  \Big\{\langle
\fa(\Ra\Xta(s))-\fb(\Rb\Xtb(s)),\Ra\Xta(s)-\Rb\Xtb(s)\rangle\\
&\qquad+\tsa\Big(
\|\fa(\Ra\Xta(s))\|^2+\|\fb(\Rb\Xtb(s))\|^2 \\
&\qquad \qquad\qquad +\|A\Ra\Rb \Xta(s)\|^2
+\|A\Ra\Rb \Xtb(s)\|^2\Big)\Big\}\,\dd s.
\end{aligned}
\end{equation}
Thus, from \eqref{eq:apf} and \eqref{eq:I2bound}, we have
\begin{equation*}
\begin{aligned}
I_2^+-I_2&\leq \int_0^t\Big\{
\langle \fa(\Ra \Xta(s))-\fb(\Rb \Xtb(s)),\Ja \Ra \Xta(s)-\Jb
\Rb \Xtb(s)\rangle\\
&\quad
-\langle\fa(\Ra\Xta(s))-\fb(\Rb\Xtb(s)),\Ra\Xta(s)-\Rb\Xtb(s)\rangle\\
&\qquad + C\Big(\|\Xta(s)-\Xtb(s)\|^2 +\tsa\big(
\|\fa(\Ra\Xta(s))\|^2+\|\fb(\Rb\Xtb(s))\|^2 \\
&\quad \qquad+\|A\Ra\Rb \Xta(s)\|^2
+\|A\Ra\Rb \Xtb(s)\|^2\big)\Big) \Big\}\, \dd s.
\end{aligned}
\end{equation*}
Using \eqref{eq:IJF} again, it follows that
\begin{equation*}
\begin{aligned}
& \langle \fa(\Ra \Xta(s))-\fb(\Rb \Xtb(s)),\Ja \Ra \Xta(s)-\Jb
\Rb \Xtb(s)\rangle
\\
&\qquad
-\langle\fa(\Ra\Xta(s))-\fb(\Rb\Xtb(s)),\Ra\Xta(s)-\Rb\Xtb(s)\rangle\\
&\quad =-\langle \fa(\Ra \Xta(s))-\fb(\Rb \Xtb(s)),\tsa  f(\Ja \Ra
\Xta(s))-\tsb f(\Jb \Rb \Xtb(s))\rangle\\
&\quad \leq C\tsa\Big( \|\fa(\Ra \Xta(s))\|^2 + \|\fb(\Rb \Xtb(s))\|^2
\Big).
\end{aligned}
\end{equation*}
Inserting this into the previous inequality, taking the supremum in
time and then the expectation, after some trivial estimates we may
conclude that
\begin{equation}
\begin{aligned}\label{eq:I2}
\EE\sup_{0\leq t\leq T}|I_2^+-I_2|
&\leq C\Big(\int_0^T\EE
\sup_{0\leq s\leq t}\|\Xta(s)-\Xtb(s)\|^2\,\dd t
\\
&\quad +\tsa \int_0^T\EE\big( \|
\fa(\Ra\Xta(s))\|^2+\|\fb(\Rb\Xtb(s))\|^2 \\
&\qquad+\|A\Ra\Rb \Xta(s)\|^2
+\|A\Ra\Rb \Xtb(s)\|^2\big)\,\dd s\Big).
\end{aligned}
\end{equation}

The term $I_1$ can be dealt with in a similar fashion. Indeed, we have that
\begin{equation*}
\begin{aligned}
I_1^+&:=\int_0^t\langle
\Ma\Ra^2\Xta(s)-\Mb\Rb^2\Xtb(s), \Aa\Xta(s)-\Ab\Xtb(s) \rangle\,\dd s\\
&=\int_0^t\|\Ma\Ra^2\Xta(s)-\Mb\Rb^2\Xtb(s)\|_1^2\,\dd s\geq 0.
\end{aligned}
\end{equation*}
Also, note that
$
I-\Mk\Rk^2=\frac{\ts}{4}A(3+\ts A)\Rk^2.$
Thus, abbreviating $\tMk=(3+kA)$, we have
\begin{equation*}
\begin{aligned}
I_1^+-I_1&=\tfrac14\int_0^t\langle
A\big(\tsa\tMa\Ra^2\Xta(s)-\tsb\tMb\Rb^2
\Xtb(s)\big),\Aa\Xta(s)-\Ab\Xtb(s)\rangle\,\dd s\\
&\leq  C \int_0^t \big(\|\Aa\Xta(s)\|+\|\Ab\Xtb(s)\|\big)\big( \tsa\|A(\tMa\Ra^2\Xta(s)\|\\
&\quad+\tsb\|A\tMb\Rb^2\Xtb(s)\|\big)\,\dd s.
\end{aligned}
\end{equation*}
But $\|\tMk x\|\leq 4\|\Mk x\|$ and hence
\begin{equation}\label{eq:I1}
\EE \sup_{0\leq t\leq T}|I_1^+-I_1|\leq C \tsa\int_0^T\EE\Big( \|\Aa\Xta(s)\|^2+\|\Ab\Xtb(s)\|^2\Big)\,\dd s.
\end{equation}
Combining the results in \eqref{eq:I4}, \eqref{eq:I3}, \eqref{eq:I2},
and \eqref{eq:I1}, we thus conclude that
\begin{equation*}
\begin{aligned}
\EE\sup_{0\leq t\leq T}\tfrac12\|\Xta(t)-\Xtb(t)\|^2 &=\EE\sup_{0\leq
  t\leq T}\big(-I_1-I_2+I_3+I_4\big)\\
&\leq\EE\sup_{0\leq t\leq T}\big(
I_1^+-I_1+I_2^+-I_2+I_3+I_4\big)\\
&\leq C\Big\{\int_0^T \EE\sup_{0\leq t\leq
  s}\|\Xta(t)-\Xtb(t)\|^2\,\dd s\\
&\quad +\tsa\Big(1+\int_0^T
\|\fa(\Ra\Xta(s))\|^2+\|\fb(\Rb\Xtb(s))\|^2\\
&\quad+\|A\Ra\Rb \Xta(s)\|^2
+\|A\Ra\Rb \Xtb(s)\|^2\\
&\quad +\|\Aa\Xta(s)\|^2+\|\Ab\Xtb(s)\|^2\,\dd s\Big) \Big\}.
\end{aligned}
\end{equation*}
The second integral is bounded by Lemma~\ref{lem:tildebnd} and
Corollary~\ref{cor:intF} since we assume $\EE \|
X_0\|_1^6<\infty$.
Note that Lemma~\ref{lem:tildebnd} is applicable as, by
\eqref{eq:resol} and \eqref{eq:MkRk}, we have
$\|A\Ra\Rb x\|\le 2 \|\Aa x\|$ and $\|\Aa x\|\le \|\Aa^{1/2}x\|_1$
(similarly for the terms with $\Ab$). Thus, by Gronwall's lemma,
\begin{equation}\label{eq:dg}
\EE\sup_{0\leq t\leq T}\tfrac12\|\Xta(t)-\Xtb(t)\|^2 \leq C(T)\tsa.
\end{equation}
Therefore, it follows that there is $\Xt$ such that  $\Xtk\rightarrow
\Xt$ in $L_2(\Omega,\cF,P;C([0,T],H))$ as
$k\rightarrow 0$. But according to Lemma~\ref{lem:XaX} $\Xt$ must be
$X$. Thus the theorem follows from \eqref{eq:dg} by taking $\gamma\to 0$.
\end{proof}

\subsection{A bound for $e^2$}
We have the following result.

\begin{thm}\label{thm:splitsteperror}
If Assumption~\ref{ass:q} is satisfied with $q=18$ and $\EE \|X_0\|_2^2< \infty$,
then there is 
\begin{align*}
C=C(\ts_0,T,\HSAQ,\EE\|X_0\|^{18}_1,\EE \|X_0\|_2^2)>0  
\end{align*}
such that
\begin{equation*}
\EE\sup_{0\leq t\leq T}\|\Xtk(t)-\Xh(t)\|^2 \leq C\ts.
\end{equation*}

\end{thm}
\begin{proof}
First note that
\begin{equation*}
\dd
(\Xtk(t)-\Xh(t))
=-\big(\Ak(\Xtk(t)-\Xb(t))+(\Fk(\Xtk(t))-\Fk(\Xb(t)))\big)\,\dd t.
\end{equation*}
Thus, since $\Xtk(0)=\Xb(0)$, it follows that
\begin{equation}\label{eq:ssito}
\begin{aligned}
\tfrac12\|\Xtk(t)-\Xh(t)\|^2&=-\int_0^t\langle \Xtk(s)-\Xh(s),
\Ak\big(\Xtk(s)-\Xb(s)\big)\rangle\,\dd s\\
&\quad-\int_0^t\langle \Xtk(s)-\Xh(s),
\Fk(\Xtk(s))-\Fk(\Xb(s))\rangle\,\dd s.
\end{aligned}
\end{equation}
We have that, for any $\epsilon_1>0$ there is $C_1$  such that
\begin{equation*}
\begin{aligned}
&-\langle \Xtk(s)-\Xh(s),
\Ak\big(\Xtk(s)-\Xb(s)\big)\rangle\\
&\quad =-\langle \Xtk(s)-\Xh(s),
\Ak\big(\Xtk(s)-\Xh(s)\big)\rangle-\langle \Xtk(s)-\Xh(s),
\Ak\big(\Xh(s)-\Xb(s)\big)\rangle\\
&\quad \leq -\|\Mk^{1/2}\Rk( \Xtk(s)-\Xh(s) ) \|_1^2 +\epsilon_1\|
\Mk^{1/2}\Rk(\Xtk(s)-\Xh(s))\|_1^2\\
&\qquad+C_1\|A^{1/2}M_k^{1/2}\Rk( \Xb(s)-\Xh(s)) \|^2.
\end{aligned}
\end{equation*}

Similarly, using \eqref{eq:onesided1},  \eqref{eq:local},
\eqref{eq:Jklip} and \eqref{eq:Mkco} we compute
\begin{equation*}
\begin{aligned}
&-\langle \Xtk(s)-\Xh(s),
\Fk(\Xtk(s))-\Fk(\Xb(s))\rangle\\
&\quad=-\langle \Xtk(s)-\Xh(s),
\Fk(\Xtk(s))-\Fk(\Xh(s))\rangle\\
&\quad\quad-\langle \Xtk(s)-\Xh(s),
\Fk(\Xh(s))-\Fk(\Xb(s))\rangle\\
&\quad \leq C\| \Rk(\Xtk(s)-\Xh(s))\|^2\\
&\quad\quad-\langle A^{1/2}\Rk\big(  \Xtk(s)-\Xh(s)\big),
A^{-1/2}\big( f(\Jk\Rk\Xh(s))-f(\Jk\Rk\Xb(s))\big)\rangle\\
&\quad\leq  C\| \Rk(\Xtk(s)-\Xh(s))\|^2+\epsilon_2 \|A^{1/2}\Rk\big(
\Xtk(s)-\Xh(s)\big)\|^2\\
&\quad \quad +C_2\|A^{-1/2}
f(\Jk\Rk\Xh(s))-f(\Jk\Rk\Xb(s))\|^2\\
&\quad\leq   C\| \Rk(\Xtk(s)-\Xh(s))\|^2+\epsilon_2 \|\Rk\big(
\Xtk(s)-\Xh(s)\big)\|_1^2  \\
&\quad \quad +C_3\big(\|\Jk\Rk\Xh(s)\|_1^4+
\|\Jk\Rk\Xb(s)\|_1^4+1\big) \|\Jk\Rk\Xh(s) -\Jk\Rk\Xb(s)\|^2\\
&\quad\leq   C\| \Xtk(s)-\Xh(s)\|^2+\epsilon_2 \|M^{1/2}\Rk\big(
\Xtk(s)-\Xh(s)\big)\|_1^2 \\
&\quad \quad +C_3\big(\|\Jk\Rk\Xh(s)\|_1^4+ \|\Jk\Rk\Xb(s)\|_1^4+1\big) \|\Xh(s) -\Xb(s)\|^2.
\end{aligned}
\end{equation*}
Taking $\epsilon_1+\epsilon_2=1$ and inserting these estimates into
\eqref{eq:ssito} we get, after taking first the supremum in time and
then the expectation, that
\begin{equation}
\label{eq:gronset}
\begin{aligned}
&\EE\sup_{0\leq t\leq T}\|\Xtk(t)-\Xh(t)\|^2
\leq C\Big(\int_0^T\EE\sup_{0\leq
  s\leq t}\|\Xtk(s)-\Xh(s)\|^2\,\dd t\\
&\qquad+\EE\int_0^T
\|A^{1/2}M_k^{1/2}\Rk( \Xb(s)-\Xh(s)) \|^2\,\dd s
\\
&\qquad+\EE\int^T_0\big( \|\Jk\Rk\Xh(s)\|_1^4+
\|\Jk\Rk\Xb(s)\|_1^4+1\big)
\|\Xh(s)
-\Xb(s)\|^2\, \dd s\Big).
\end{aligned}
\end{equation}
By H\"older's inequality with exponents $3/2$ and $3$, the linear
growth bound \eqref{eq:Jklip1} of $\Jk$ and the contarctivity of $\Rk$
in $\dH^1$,  the last term is bounded by
\begin{equation*}
\begin{aligned}
&C\Big(\EE \int^T_0 \big(\|\Xh(s)\|_1^6+
\|\Xb(s)\|_1^6+1\big)\,\dd s\Big)^{2/3}
\Big(\EE\int_0^T \big(\|\Xh(s)-\Xb(s)\|^6\big)\,\dd s\Big)^{1/3}
\\ &\qquad
\leq C(\EE\|X_0\|^{18}_1, \EE \|X_0\|_2^2,T,\ts_0,\HSAQ)\ts.
\end{aligned}
\end{equation*}
The last inequality follows since the first integral is bounded using Lemma~\ref{lem:discretebounds} and
Corollary~\ref{cor:jkrkxhat} and the second integral is bounded by
$C\ts$ using Proposition~\ref{prop:supEerr}. As $\Mk^{1/2}\Rk$ is a
bounded operator (uniformly in $k$) that commutes with $A^{1/2}$,
Proposition~\ref{prop:Eoneerr} asserts the same bound on the second
term in the right hand side of \eqref{eq:gronset} and thus
\begin{equation}
\EE\sup_{0\leq t\leq T}\|\Xtk(t)-\Xh(t)\|^2\leq C\Big(\ts+\int_0^T\EE\sup_{0\leq
  s\leq t}\|\Xtk(s)-\Xh(s)\|^2\,\dd t\Big),
\end{equation}
and the statement of the theorem follows by Gronwall's lemma.
\end{proof}
\begin{rem}
  Due to the fact that we have assumed additive noise, the noise term
  vanishes upon taking the difference between $\Xh$ and $\Xt$. Thus
  It\^o's formula is not needed in the proof of
  Theorem~\ref{thm:splitsteperror}. However, if multiplicative noise
  would have been considered and evaluated explicitly, the use of the
  fundamental theorem of calculus in \eqref{eq:ssito} could be
  replaced by the use of It\^o's formula since the corresponding
  construction of $\Xh$ will be adapted and continuous, and, under
  suitable assumptions on the covariance operator, similar results may
  be achieved. For the BE scheme in \eqref{eq:be}, however, a
  continuous adapted interpolation cannot be constructed.  See
  \cite{HMS} for the finite dimensional case.
\end{rem}

\subsection{Convergence of the backward Euler scheme}\label{sec:euler}

Finally, we show that the backward Euler scheme, \eqref{eq:be},  for \eqref{eq:sac} is an
$\Ordo(k^{1/2})$ perturbation of the split-step scheme
\eqref{eq:ss10}--\eqref{eq:ss1c}, thus converges with the same rate to
the solution of the stochastic Allen--Cahn equation.

First note that the elliptic equation
\begin{equation}\label{eq:bescheme}
x+k(Ax+f(x))=y
\end{equation}
has a unique weak solution
$x\in \dH^1$ for every $y\in \dH^1$ if $2k\beta^2<1$. To see this observe first that for any $K\in \R$
the function
\begin{equation}\label{eq:L}
L(\nabla u,u,\xi)=\tfrac12 k
|\nabla u|^2+\tfrac12{(1-k\beta^2)}|u|^2+\tfrac14{k} u^4-u y(\xi)+Ky(\xi)^2
\end{equation}
is a Lagrangian for \eqref{eq:bescheme} and if $K$ is large enough,
then it even satisfies the coercivity condition $L(\eta,\nu,\xi)\geq
C|\eta|^2$. Following \cite{MR2597943}, Chapter 8.2, Theorems 1, 2, and
4, the functional 
\begin{equation*}
I(v)=\int_{\cD}L(\nabla v,v,\xi)\,\dd \xi
\end{equation*}
has a minimiser $u\in W^{1,4}_0$ which is a weak solution of
\eqref{eq:bescheme} with test functions $v\in W^{1,4}_0$. Then
$u\in \dot{H}^1$ and using first test functions $v\in C^{\infty}_c$ in
the weak form of \eqref{eq:bescheme} and a standard approximation
argument one easily sees that $u$ is a weak solution with test
functions $v\in \dot{H}^1$. We note that Theorems 1 and 2 in
\cite{MR2597943} assume smoothness of $L$, which does not hold in our
case in the variable $\xi$ as $y$ only belongs to $\dot{H}^1$. However
given the simple and explicit dependence of $L$ on $\xi$ one verifies
that the proofs of Theorems 1 and 2 hold verbatim (in places even
using simpler arguments) for $L$ given by \eqref{eq:L}. Finally,
uniqueness of weak solutions \eqref{eq:bescheme} follows as in the
uniqueness part of the proof of Lemma~\ref{lem:fprop} from
\eqref{eq:onesided1} using also the positivity of $A$.

Hence, if $X_0$ and $\Delta W^{j-1}$ belongs to $\dH^1$ almost surely,
the BE scheme has a unique pathwise solution almost surely.

We need two results on H\"older regularity of $\Xj$.

\begin{lem}\label{prop:supHolder}
  If, for some $p\ge 1$, Assumption~\ref{ass:q} is satisfied with
  $q=6p$, then there exists 
  \begin{align*}
  C=C(k_0,p,T,\HSAQ,\EE\|X_0\|_1^{6p})>0 
  \end{align*}
  such that
\begin{equation*}
\sup_{1\leq j\leq N}\EE\|\Xj-\XX^{j-1}\|^{2p}\leq C\ts^p.
\end{equation*}
\end{lem}

\begin{lem}\label{prop:EoneHolder}
  If Assumption~\ref{ass:q} is satisfied with $q=6$ and
  $\EE\|X_0\|_2^2<\infty$, then there is
  \begin{align*}
  C=C(k_0,T,\HSAQ,\EE\|X_0\|_1^{6},\EE\|X_0\|_2^{2})>0
  \end{align*}
such that
\begin{equation*}
\sum_{j=1}^N\ts\EE\|\Xj-\XX^{j-1}\|_1^2\leq C\ts.
\end{equation*}
\end{lem}

The proofs of Lemmata~\ref{prop:supHolder} and \ref{prop:EoneHolder}
are omitted as they are analogous to the proofs of
Propositions~\ref{prop:supEerr} and \ref{prop:Eoneerr}, respectively,
noting that
$\Xj-\XX^{j-1}=-k\Ak\Xj-k\Fk(\Xj)+\int_{t_{j-1}}^{t_j}\Rk\,\dd W(s)$.

In the error analysis below we will restrict ourselves to the discrete grid for brevity. It is
worth noting that the proof of the following theorem is the fully
discrete analog of the proof of Theorem~\ref{thm:splitsteperror}.

\begin{thm}\label{thm:BESSerr}
  Let $\Xbej$ and $\Xj$ be the solution of \eqref{eq:be} and the
  system \eqref{eq:ss10}--\eqref{eq:ss1c}, respectively. Under the
  same assumptions as in Theorem~\ref{thm:splitsteperror} there exists
  $C>0$, depending on the same data as the constant in
  Theorem~\ref{thm:splitsteperror}, such that
\begin{equation}\label{eq:BESSerr}
\EE\sup_{0\leq j\leq N}\|\Xbej-\Xj \|^2\leq C\ts.
\end{equation}
\end{thm}

\begin{proof}
We first note that the split-step scheme may be written as
\begin{equation}\label{eq:ssnew}
  \begin{aligned}
\Xj-\Xjm+\ts A\Xj
-\hts A(\Xj-\Xjm+\hts A\Rk \Xjm)
+\ts \fk(\Rk \Xjm)=\Delta W^{j-1},
  \end{aligned}
\end{equation}
for  $j=1,2\ldots,N$ with $\XX^0=X_0$.
This is easiest seen by multiplying both sides in the first equality
of \eqref{eq:onestep}  by $I+\tfrac
k2A=\Rk^{-1}$, using that $\Rk \Xjm=(I-\tfrac k2A(I-\tfrac k2
A\Rk))\Xjm$ and then rearranging. We leave the details to
the reader.

Subtracting \eqref{eq:ssnew} from \eqref{eq:be}, taking inner
products with $\Xbej-\Xj$ and rearranging yields
\begin{equation}\label{eq:bigbe}
\begin{aligned}
&\tfrac12\big(\|\Xbej-\Xj \|^2-\|\Xbe^{j-1}-\Xjm
\|^2+\|\Xbej-\Xj-\Xbe^{j-1}+\Xjm\|^2 \big)
+
k\|\Xbej-\Xj \|_1^2\\
&\quad =\frac k2\langle \Xj-\Xjm,\Xbej-\Xj
\rangle_1+\frac {k^2}4\langle A\Rk\Xjm,\Xbej-\Xj\rangle_1\\
&\qquad -k\langle
f(\Xbej)-f(\Xj),\Xbej-\Xj \rangle\\
&\qquad-k\langle A^{-1/2}
(f(\Xj)-\fk(\Rk X^{j-1})),A^{1/2}(\Xbej-\Xj) \rangle\\
&\quad=:k(e^b_1+e^b_2+e^b_3+e^b_4).
\end{aligned}
\end{equation}
It is easily seen that
\begin{align}
|e_1^b|&\leq C\| \Xj-\Xjm\|_1^2+\epsilon\|\Xbej-\Xj\|_1^2,\label{eq:ebe1}\\
|e_2^b|&\leq Ck^4\| A\Rk\Xjm\|_1^2+\epsilon\|\Xbej-\Xj\|_1^2,\\
e_3^b&\leq \beta^2\| \Xbej-\Xj\|^2, \label{eq:ebe3}\\
\label{eq:ebe4}
|e_4^b|&\leq C\| A^{-1/2} ( f(\Xj)-f(\Jk\Rk X^{j-1}))\|^2+\epsilon
\|\Xbej-\Xj\|_1^2.
\end{align}
By \eqref{eq:local}, the linear growth bound \eqref{eq:Jklip1} of
$\Jk$ , and the contarctivity of $\Rk$ in $\dH^1$, we have 
\begin{equation}\label{eq:lipbe}
  \begin{aligned}
\| A^{-1/2} ( f(\Xj)-f(\Jk\Rk\Xjm))\|
\leq C\Big(\|\Xj \|_1^2+\|\Xjm\|_1^2 +1\Big)\|\Xj-\Jk\Rk\Xjm\|.
  \end{aligned}
\end{equation}
From the triangle inequality, \eqref{eq:yosidacomp}, \eqref{eq:AhRk}
and \eqref{eq:fgrowth}  it follows that
\begin{equation}\label{eq:yosbe}
\begin{aligned}
\|\Xj-\Jk\Rk\Xjm\|&\leq \|\Xj-\Xjm\|+k\|A\Rk\Xjm\|+k\|\fk(\Rk\Xjm)\|\\
&\leq \|\Xj-\Xjm\|+k^{1/2}\|\Xjm\|_1+k(\|\Xjm\|_1^3+1).
\end{aligned}
\end{equation}
Thus, from \eqref{eq:ebe4}--\eqref{eq:yosbe} we get
\begin{equation}\label{eq:myeq}
  \begin{aligned}
|e_4^b|&\leq C\Big(P_4(\|\Xj
\|_1,\|\Xjm\|_1 ) \|\Xj-\Xjm\|^2+\ts P_{10}(\|\Xjm\|,\|\Xj
\|)\Big)\\
& \quad
+ \epsilon \|\Xbej-\Xj\|_1^2
  \end{aligned}
\end{equation}
with $P_4$ and $P_{10}$ being a positive polynomials of degree 4 and
10, respectively.  Inserting the bounds in
\eqref{eq:ebe1}--\eqref{eq:ebe3} and \eqref{eq:myeq} into
\eqref{eq:bigbe} we conclude that
\begin{equation*}
\begin{aligned}
&\tfrac12\Big(\|\Xbej-\Xj \|^2- \|\Xbe^{j-1}-\Xjm
\|^2\Big) +k\|\Xbej-\Xj \|_1^2
\\ & \qquad 
\leq k\Big\{3 \epsilon\|\Xbej-\Xj
\|^2_1+\beta^2
\|\Xbej-\Xj\|^2 +C\big(\| \Xj-\Xjm\|_1^2+\ts^3\| A\Rk\Xjm\|_1^2
\\ & \qquad\quad +P_4(\|\Xj
\|_1,\|\Xjm\|_1^4 ) \|\Xj-\Xjm\|^2+\ts P_{10}(\|\Xjm\|_1,\|\Xj
\|_1)\big)\Big\}.
\end{aligned}
\end{equation*}
By \eqref{eq:AhRk} we have that $k\| A\Rk\Xjm\|_1^2\leq \|
A\Xjm\|^2$.
Choosing $\epsilon=\tfrac13$, we get, after summing up from $j=1$ to
$j=N$ and taking first the supremum over the discrete times and then
the expectation, that
\begin{equation}\label{eq:becomp}
\begin{aligned}
&\EE\sup_{0\leq n\leq N}\tfrac12\|\Xbe^n-\XX^n \|^2
\leq  \EE\sum_{j=1}^N
k\beta^2 \|\Xbej-\Xj\|^2\\
&\qquad+Ck\Big\{k\sum_{j=1}^N  \EE \big( \| A\Xjm\|^2+P_{10}(\|\Xjm\|,\|\Xj
\|)\big)\Big\}\\
&\qquad
+ C\sum_{j=1}^N k\EE\| \Xj-\Xjm\|_1^2
+C\EE\sum_{j=1}^N kP_4(\|\Xj
\|_1,\|\Xjm\|_1 ) \|\Xj-\Xjm\|^2.
\end{aligned}
\end{equation}
From Lemma~\ref{lem:discretebounds} it follows that there is a
constant $C_1>0$ depending on $\ts_0$, $t_N$, $\EE\|X_0\|_1^{10}$,
$\EE\|X_0\|_2^2$, and $\HSAQ$ that bounds the two sums in
the parenthesis in the right hand side of \eqref{eq:becomp}.  The next
term is bounded in Lemma~\ref{prop:EoneHolder}.
For the last sum we use H\"older's inequality
with exponents $5/2$ and $5/3$ to conclude that
\begin{equation}\label{eq:becomp2}
\begin{aligned}
&k\EE\sum_{j=1}^N P_4(\|\Xj
\|_1,\|\Xjm\|_1) \|\Xj-\Xjm\|^2\\
&\qquad\leq
\Big(\EE\sum_{j=1}^N kP_{10}(\|\Xj
\|_1,\|\Xjm\|_1)\Big)^{2/5}\Big( \EE\sum_{j=1}^N
k \|\Xj-\Xjm\|^{10/3}\Big) ^{3/5}.
\end{aligned}
\end{equation}
By Lemmata~\ref{lem:discretebounds} and \ref{prop:supHolder} there
exists a $C_2$ depending only on $k_0$, $t_N$, $\EE\|X_0\|^{10}_1$, and
$\HSAQ,$ such that the expression on the right hand side in \eqref{eq:becomp2}
is bounded by $ C_2\ts$.
Thus, for some $C>0$, we have that
\begin{equation}
\begin{aligned}
\EE\sup_{0\leq n\leq N}\|\Xbe^n-\XX^n \|^2&\leq  \sum_{j=1}^N
2k\beta^2 \EE\sup_{0\leq i\leq j}\|\Xbe^i-\XX^i\|^2+C\ts
\end{aligned}
\end{equation}
and, since $2k\beta^2< 1 $, the claim follows by
Lemma~\ref{lem:gronwallDiscrete}.
\end{proof}
\begin{rem}
The proof of Theorem~\ref{thm:BESSerr} gives a little bit more than
what is stated in the theorem. Only $\EE\|X_0\|_1^{10}$ needs to be
bounded and not $\EE\|X_0\|_1^{18}$.
\end{rem}

\subsection*{Acknowledgement}
  M.~Kov\'acs was supported by Marsden Fund project number
  UOO1418. F.~Lindgren was supported by JSPS KAKENHI grant number
  15K45678 and STINT, the Swedish Foundation for International
  Cooperation in Research and Higher Education.  The authors thank the
  anonymous referees for helpful criticism. 

\def\cprime{$'$}
\providecommand{\WileyBibTextsc}{}
\let\textsc\WileyBibTextsc
\providecommand{\othercit}{}
\providecommand{\jr}[1]{#1}
\providecommand{\etal}{~et~al.}

\end{document}